\RequirePackage{fix-cm}
\documentclass[smallextended,numbook,runningheads]{svjour3}
\smartqed  
\usepackage{verbatim}
\usepackage{xcolor}
\usepackage{graphicx}
\usepackage{amsmath,amssymb,bm}
\usepackage{graphicx,graphics,color}
\definecolor{refgreen}{rgb}{0,0.5,0}
\usepackage[colorlinks,linkcolor=blue,citecolor=blue]{hyperref}
\usepackage{tikz}
\usepackage{booktabs,subfig} 

\usepackage[utf8]{inputenc}
\journalname{...}
\date{ \phantom{b} \vspace{45mm}\phantom{e}}
\def\makeheadbox{\hspace{-4mm}Version of \today}

\usepackage{color}

\allowdisplaybreaks

\def\half{\tfrac{1}{2}}

\newcommand\bfd{{\mathbf d}}
\newcommand\bfe{{\mathbf e}}

\newcommand\bff{{\mathbf f}}

\newcommand\bfn{{\mathbf n}}

\newcommand\bfu{{\mathbf u}}
\newcommand\bfv{{\mathbf v}}
\newcommand\bfw{{\mathbf w}}
\newcommand\bfx{{\mathbf x}}

\newcommand\bfA{{\mathbf A}}

\newcommand\bfH{{\mathbf H}}

\newcommand\bfM{{\mathbf M}}

\newcommand\bfzero{{\mathbf 0}}

\newcommand\calO{{\mathcal O}}

\newcommand\andquad{\quad\hbox{ and }\quad}

\renewcommand{\d}{\textnormal{d}}

\newcommand{\wt}{\widetilde}

\newcommand{\Ga}{\varGamma}

\newcommand{\nbg}{\nabla_{\varGamma}}

\newcommand{\mat}{\partial^{\mbox{\tiny$\bullet$}}}

\newcommand{\diff}{\frac{\d}{\d t}}
\newcommand{\eps}{\varepsilon}

\newcommand{\nb}{\nabla}

\newcommand{\pa}{\partial}
\newcommand{\R}{\mathbb{R}}

\def \to {\rightarrow}

\newcommand{\vphi}{\varphi}

\newcommand{\bfvb}{\bar\bfv}

\newcommand{\wtu}{\widetilde{\bfu}}
\newcommand{\wtv}{\widetilde{\bfv}}
\newcommand{\wtx}{\widetilde{\bfx}}

\newcommand{\xs}{\bfx^\ast}

\newcommand{\uls}{\bfu_\ast}
\newcommand{\vls}{\bfv_\ast}

\newcommand{\xls}{\bfx_\ast}
\newcommand{\nls}{\bfn_\ast}
\newcommand{\Hls}{\bfH_\ast}

\newcommand{\wtxls}{\widetilde{\bfx}_\ast}

\newcommand{\eu}{\bfe_\bfu}
\newcommand{\ev}{\bfe_\bfv}
\newcommand{\evb}{\bar\bfe_{\bfv}}
\newcommand{\ex}{\bfe_\bfx}

\newcommand{\du}{\bfd_\bfu}
\newcommand{\dv}{\bfd_\bfv}

\newcommand{\dx}{\bfd_\bfx}

\newcommand{\dof}{N}

\newcommand{\doi}[1]{doi:\href{https://doi.org/#1}{\texttt{#1}}}

\newcommand{\be}{\partial^\tau\!}

\newcommand{\area}{\mathcal{E}}
\newcommand{\bflambda}{\boldsymbol{\lambda}}

\def\ecl{\color{black}}

\begin{document}

\title{Structure-preserving Lagrange-multiplier methods for mean curvature flow and their error bounds}

\titlerunning{Lagrange-multiplier FEMs for curvature flows}        

\author{Bal\'{a}zs Kov\'{a}cs \\
Buyang Li\\
Christian Lubich}

\authorrunning{B. Kov\'{a}cs, B. Li, and C. Lubich} 

\institute{B.~Kov\'{a}cs: Institute of Mathematics, Paderborn University, Warburgerstr.~100., 33098 Paderborn, Germany.\\ 
B. Li: Department of Applied Mathematics, The Hong Kong Polytechnic University, Hong Kong. \email{buyang.li@polyu.edu.hk}\\ 
C. Lubich: Mathematisches Institut, Universität Tübingen, Auf der Morgenstelle 10, D-72076 Tübingen, Germany.\\
}

\date{\today}

\maketitle

\begin{abstract}
We propose and analyze a class of evolving surface finite element methods for mean curvature flow of closed surfaces using Lagrange multipliers for preserving the energy-decreasing structure. The algorithm is based on the solution-driven formulation using Huisken's evolution equations for the normal vector and the mean curvature. The time
discretizations use linearly implicit backward difference formulas (BDF) for the
parabolic geometric variables and implicit Adams updates for the nodal
positions. The approach also accommodates artificial tangential velocities of
minimal-deformation-rate type. The resulting fully discrete algorithms are area-decreasing at every time step, with a prescribed decay rate determined by the computed mean curvature.
We prove local existence and uniqueness of the discrete Lagrange multiplier and convergence of a simplified Newton iteration for its computation under weak regularity assumptions. Under  stronger regularity assumptions, as used in the convergence theory for the underlying evolving surface finite element method, we derive optimal-order error bounds of order $h^k+\tau^q$ in the $H^1$-norm for finite elements of polynomial degree $k\ge 2$ and $q$-step BDF and $q$-step implicit Adams methods with $2\le q\le 5$, both without and with minimal-deformation-rate tangential motion.  Numerical experiments for mean curvature flow of a sphere confirm the predicted convergence rates and show that the Lagrange-multiplier correction entails
only a small computational overhead that is essentially independent of the mesh size and the time step size.
\end{abstract}

\tableofcontents

\section{Introduction}

Surface evolution under geometric flows, including mean curvature flow, Willmore flow, and surface diffusion flow, has been intensively studied in the last decades \cite{Randolph-et-2007,Thompson-2012,Jiang-et-2019,Nitschke-Reuther-Voigt-2020,Torres-Sanchez-2019} both from the viewpoint of applications and mathematical analysis. 
However, approximating surface evolution under geometric flows by numerical methods remains a challenging task. Numerical approximation of the mean curvature flow was first addressed by Dziuk \cite{Dziuk1990} in 1990 through a finite element method (FEM) with moving nodes that determine the approximate surface. 
In the literature, such FEMs appear under various synonymous names: Evolving surface FEM, evolving FEM, and parametric FEM. Here we will often just speak of FEM for brevity.
Many researchers have adopted and extended  FEMs to various geometric flows of curves and surfaces, including:
\begin{itemize}
    \item The development of evolving surface FEMs by Dziuk and Elliott \cite{DziukElliott_ESFEM}.
    \medskip
    \item The development of parametric FEMs with re-meshing techniques by B\"ansch, Morin, and Nochetto \cite{Bansch2005}.
    \medskip
    \item The introduction of tangential motions to improve the mesh quality (i.e., mesh points distribution and shape of triangles) of computed surfaces by Barrett, Garcke, and N\"urnberg \cite{Barrett2007,Barrett2008}. 
    \medskip
    \item The development of mesh quality improving tangential motions using harmonic map heat flow and the DeTurck trick \cite{Elliott2016,Elliott-Fritz-2017}. Semi-discrete error estimates for a DeTurck-based FEM discretisation were recently shown in \cite{DeckelnickStyles2026}.
    \medskip
    \item The development of tangential motions which make the evolving surfaces have minimal deformation rate (MDR) \cite{Hu-Li-2022}.
    \medskip
    \item The development of tangential motions ensuring minimal deformation from the initial surface \cite{DL-2024}.
        \medskip
    \item The development of tangential motions via a dual formulation of the minimal deformation rate (dual-MDR) scheme \cite{GaoLiTang2026}.
\end{itemize}

The first rigorous proof of convergence for evolving FEMs applied to mean curvature flow on closed surfaces was established in \cite{MCF}. This approach involves reformulating the mean curvature flow into an equivalent system of solution-driven surface evolution using Huisken's evolution equations for the normal vector and mean curvature \cite{Huisken1984}, as well as analyzing the error of the evolving FEM using the matrix-vector framework and its properties established in \cite{DLM2012,KLLP2017}. This approach was later extended to forced mean-curvature flow \cite{KLL2020,EGK2022} and further to Willmore flow and surface diffusion \cite{Willmore}, reformulating Willmore flow and surface diffusion as a solution-driven surface evolution and proving the convergence of the FEM for this equivalent formulation. These advances of numerical analysis also allowed proving the convergence of Dziuk's original semi-implicit FEM (with finite elements of degree $k\ge 3$) for mean curvature flow \cite{Bai-Li-2024}, as well as the convergence of a stabilized FEM incorporating artificial tangential motions of the Barrett--Garcke--N\"urnberg type for curve shortening flow \cite{Bai-Li-2025}. 

Despite these advances in the numerical analysis of mean curvature flow, existing algorithms with proven convergence so far lack at least one of the following desirable features: (i) \emph{energy stability} and (ii) an explicit \emph{tangential motion} mechanism for improving
the quality of the evolving mesh. Both properties have been achieved by several alternative approaches, for example:
\begin{itemize}
  \item Energy-stable FEMs augmented with artificial tangential motion of the
  Barrett--Garcke--N\"urnberg type to improve mesh quality (e.g., node distribution and
  triangle shape) on computed surfaces; see
  \cite{Bao2021,Barrett2007,Barrett2008}.
  \medskip
  \item Energy-stable FEMs based on Lagrange-multiplier formulations
  \cite{GJSZ-2025,Gao-Li-2025}, including tangential motions of the Barrett--Garcke--N\"urnberg type and of minimum-deformation-rate type.
  \medskip
  \item High-order energy-stable algorithm with artificial tangential motion based on a dual-MDR formulation \cite{GaoLiTang2026} for both mean curvature flow and surface diffusion.
  \medskip
  \item Structure-preserving high-order continuous-in-time Petrov–Galerkin discretisation of geometric surface flows \cite{ZhangAndrewsFarrell2026}. This approach uses auxiliary variables and is provably energy-stable and volume-conserving.
\end{itemize} 
However, a rigorous convergence and error analysis for these methods remains open and challenging.

In this paper, we develop a rigorous error analysis for energy-stable FEMs based on a Lagrange multiplier approach, building upon the existing error analysis of evolving FEMs in \cite{MCF}. This approach can be combined with incorporating tangential motions to improve the mesh quality, building additionally on \cite{Hu-Li-2022}. We introduce new frameworks for the numerical methods and geometric flows and integrate them into the existing error analysis. In this way our results provide a unified theoretical foundation for understanding and advancing numerical methods in this field. 

While we present the Langrange multiplier approach to preserving energy decay only for mean curvature flow, the approach could be similarly used to obtain a fully discrete energy-decaying method for Willmore flow and an energy-decaying and volume-preserving method for surface diffusion, allowing for a rigorous convergence analysis by combining the techniques of this paper with the error analysis in \cite{Willmore,BullerjahnKovacs2026Willmore}, and presumably this can be extended to further geometric gradient flows.

The paper is organized as follows.  In Section~2 we introduce continuous Lagrange-multiplier formulations of mean curvature flow, first without tangential motion and then with an artificial tangential velocity of minimal-deformation-rate type.  Section~3 recalls the evolving surface finite element discretization of the solution-driven formulation of mean curvature flow used in \cite{MCF}.  In Section~4 we formulate the corresponding
semi-discrete Lagrange-multiplier methods and show how the multiplier enforces the discrete area decay.  Section~5 presents the fully discrete linearly-implicit BDF and implicit Adams schemes, including the scalar nonlinear equation for the Lagrange multiplier, and states the main error estimates.  The proof for the method without tangential motion is given in Section~6, where the main additional ingredient is a stability estimate for the multiplier.  Section~7 extends the argument to the method with minimal-deformation-rate tangential
motion.  Finally, Section~8 reports implementation details and numerical experiments illustrating the convergence rates, the area decay, and the computational overhead of the Lagrange-multiplier correction.

\section{Continuous formulations with Lagrange multipliers}

In this paper, we shall develop a Lagrange multiplier approach for constructing energy-stable FEMs for mean curvature flow and related problems. Consider the evolution of a surface under mean curvature flow with initial surface \( \Gamma^0 \), and let the surface \( \Gamma[X(\cdot,t)] \) be the image of a flow map \( X(\cdot,t): \Gamma^0 \to \mathbb{R}^3 \). The original continuous reformulation of mean curvature flow that is discretized by Kov\'acs, Li, and Lubich in \cite{MCF} incorporates the evolution equations for the normal vector $n$ and mean curvature $H$ derived by Huisken \cite{Huisken1984} and is written as follows:
\begin{align}\label{KLL-form}
\begin{aligned}
    \partial_t X &= v \circ X,  \\
    v &= -Hn,  \\
    \mat n - \Delta_{\Gamma[X]} n &= |\nabla_{\Gamma[X]} n|^2 n,  \\
    \mat H - \Delta_{\Gamma[X]} H &= |\nabla_{\Gamma[X]} n|^2 H, 
\end{aligned}
\end{align}
where $\mat$ is the material derivative, i.e. 
$\mat u(x,t)= \frac{\d}{\d t}u(X(q,t),t)$ at $x=X(q,t)$ with $q\in\Gamma^0$,
and
where \( v, n \), and \( H \) denote the velocity, normal vector, and mean curvature of the evolving surface \( \Gamma[X] = \Gamma[X(\cdot,t)] \). However, direct discretization of \eqref{KLL-form} by FEMs does not preserve the energy (area) decrease in the continuous mean curvature flow. 

We introduce an additional Lagrange multiplier \( \lambda \) to this formulation, rewriting it into the following form:
\begin{subequations}\label{KLL-lambda}
\begin{align}
    \partial_t X &= (1+\lambda) \, v \circ X,  \\
    v &= -Hn , \notag \\
    \mat n - \Delta_{\Gamma[X]} n &= |\nabla_{\Gamma[X]} n|^2 n,  \\
    \mat H - \Delta_{\Gamma[X]} H &= |\nabla_{\Gamma[X]} n|^2 H,  \\
    \frac{\d}{\d t}\int_{\Gamma[X]} 1&= -\int_{\Gamma[X]} H^2 . 
    \label{eq:MCF energy decay}
\end{align}
\end{subequations}
The last equation is a constraint corresponding to the Lagrange multiplier \( \lambda \), as the exact solution of mean curvature flows satisfies the following relation:
\begin{equation} \label{area-decay}
    \frac{\d}{\d t} \int_{\Gamma[X]} 1 = \int_{\Gamma[X]} \nabla_{\Gamma[X]} \cdot v
    = \int_{\Gamma[X]} Hn \cdot v
    = -\int_{\Gamma[X]} H^2 , 
\end{equation}
where the first equality is the Leibniz formula and the second equality is obtained from
$\nbg \cdot v = \nbg \cdot (- H n) = -  \nbg H \cdot n - H \nbg \cdot n = - H^2$ ;
cf.~\cite[Corollary~3.6~(i)]{Huisken1984}.
Therefore, the exact solution of mean curvature flow, together with $\lambda=0$, satisfies \eqref{KLL-lambda}. In other words, the solution of the continuous formulation in \eqref{KLL-form} automatically satisfies \eqref{KLL-lambda} with \( \lambda = 0 \). 

The numerical solution of \eqref{KLL-form} may not satisfy the energy-decay property $\frac{\d}{\d t} \int_{\Gamma[X]} 1 \leq 0,$ while the numerical solution of \eqref{KLL-lambda} still retains this energy-decay property due to the constraint in \eqref{eq:MCF energy decay}. 
One challenge to be overcome in this paper is the convergence of FEM together with an appropriate time discretization for \eqref{KLL-lambda}, which is not obvious due to the introduction of the Lagrange multiplier and constraint. 

Moreover, we shall further consider a Lagrange multiplier formulation that incorporates an artificial tangential motion of MDR-type:
\begin{subequations}\label{KLL2-lambda}
\begin{align}
    \partial_t X &= (1+\lambda)\, v \circ X,  \\
    v \cdot n &= -H ,  \label{MDR-1}\\
    -\Delta_{\Gamma[X]} v &= -\mu n,  \label{MDR-2}\\
    \mat n - v\cdot\nabla_{\Gamma[X]}n - \Delta_{\Gamma[X]} n &= |\nabla_{\Gamma[X]} n|^2 n,  \\
    \mat H - v\cdot\nabla_{\Gamma[X]}H - \Delta_{\Gamma[X]} H &= |\nabla_{\Gamma[X]} n|^2 H,  \\
    \frac{\d}{\d t}\int_{\Gamma[X]} 1&= -\int_{\Gamma[X]} H^2 . \label{KLL2-lambda-constraint}
\end{align}
\end{subequations}
The second and third equations generate a tangential motion that minimizes the deformation rate  $$ \int_{\Gamma[X]} |\nabla_{\Gamma[X]} v|^2$$ of the evolving surface under the constraint \eqref{MDR-1}. The unknown function $\mu\in H^{-1}(\Gamma[X])$ is a Lagrange multiplier of this constrained optimization problem, which is to be determined. Discretization of \eqref{KLL2-lambda} maintains mesh quality as the surface evolves; see the motivation and discussions in \cite{Hu-Li-2022}.
Here we can naturally incorporate this MDR tangential motion into our Lagrange multiplier approach. Again, the convergence of FEMs together with an appropriate time discretization for this continuous formulation is a challenge to be addressed in this paper.

\section{Evolving FEM for mean curvature flow: method of \cite{MCF}}
\label{section:semi-discrete MCF}

In this section, we recall the evolving FEM for \eqref{KLL-form} discussed in \cite{MCF}, using the notion of evolving surface FEM described in \cite{Dziuk88,DziukElliott_ESFEM}. In this method, the given smooth initial surface $\Gamma^0$ is approximated by a piecewise curved triangulated surface that interpolates $\Gamma^0$, with each triangular piece of $\Gamma_h^0$ being the image of the reference plane triangle under a polynomial map of degree $k$; see \cite{DziukElliott_ESFEM,demlow2009higher,kovacs2018higher}. The nodal vector $\bfx^0=(x_1^0, \cdots, x_\dof^0) \in \R^{3\dof}$, which collects all the nodes $x_j^0$ $(j=1,\dots,\dof)$ of $\Gamma_h^0$, is evolved in time, and its value at time $t$ is denoted by $\bfx(t)$, determining a triangulated surface $\Gamma_h[\bfx(t)]$ that approximates $\Gamma(t)$.

The finite element spaces on $\Ga_h[\bfx(t)]$ and $\Ga_h[\bfx^*(t)]$ are denoted by $S_h[\bfx(t)]$ and $S_h[\bfx^*(t)]$, respectively. The evolving FEM for mean curvature flow discussed in \cite{MCF} essentially seeks nodal vectors $\bfx(t) \in \R^{3\dof}$, $\bfv(t)\in \R^{3\dof}$ and $\bfu(t)=(\bfn(t);\bfH(t))\in \R^{4\dof}$ which approximate the interpolated nodal vectors $\bfx^*(t)$, $\bfv^*(t)$ and $\bfu^*(t)$, respectively.

The semi-discrete evolving finite element scheme is to find nodal vectors $\bfx(t) \in \R^{3\dof}$ determining the surface $\Ga_h[\bfx(t)]$ and their velocities $\bfv(t)\in \R^{3\dof}$ and the nodal vector $\bfu(t)=(\bfn(t);\bfH(t))\in \R^{4\dof} $ of a finite element function $u_h(\cdot,t)=(n_h(\cdot,t);H_h(\cdot,t))\in S_h[\bfx(t)]^4$
satisfying the following weak formulation for all finite element test functions  ${\varphi_h}\in S_h[\bfx(t)]^4$ and all $t\in[0,T]$: omitting the argument $t$,
\begin{equation}
	\label{eq:MCF discrete}
	\begin{aligned}
		\int_{\Ga_h[\bfx]} \!\!\! \mat_h u_h\cdot \varphi_h + \int_{\Ga_h[\bfx]} \!\!\!\! \nabla_{\Ga_h[\bfx]} u_h \cdot \nabla_{\Ga_h[\bfx]} \varphi_h  = &\ \int_{\Ga_h[\bfx]} \!\!\! 
        |\nabla_{\Gamma_h[\bfx]} n_h|^2 u_h  \cdot \varphi_h  
        \\[2mm]
        \bfv = &\ -  \bfH \bullet \bfn 
        \\
        {\bf\dot\bfx} = &\ \bfv ,
	\end{aligned}
\end{equation}
where $\mat_h u_h$ is the material derivative of $u_h$ on $\Ga_h[\bfx]$, which equals the finite element function with nodal vector ${\bf\dot \bfu}=d\bfu/dt$.
Furthermore, $\bullet$ denotes the componentwise product of vectors, i.e. 
$(\bfH \bullet \bfn)_i=H_i n_i = H_h(x_i) n_h(x_i)$. We mention that in \cite{MCF} the discrete velocity was determined via a Ritz projection, but it was found later in \cite{Willmore} that the simpler construction of $\bfv$ as given above also yields a stable discretization of the same order of convergence.


For $t \in [0,T]$, the globally continuous finite element nodal basis functions which span $S_h[\bfx(t)]$ are denoted by $\phi_j[\bfx(t)] $, $j=1,\dotsc,\dof$. They are the unique functions whose piecewise pullback to the flat reference triangle is a polynomial of degree $k$ and which have the nodal values $\phi_j[\bfx(t)](x_i(t)) = \delta_{ij}$ for all $i,j = 1, \dotsc, \dof$.  

\medskip\noindent
{\bf Matrix--vector formulation.}
The surface mass matrix $\bfM(\bfx)\in \R^{\dof\times\dof}$ and the stiffness matrix $\bfA(\bfx)\in \R^{\dof\times\dof}$are defined in terms of these finite element basis functions via their $(i,j)$th entries, for $i, j = 1, \ldots, \dof$,
\begin{align}\label{def-matrix-M-A}
\begin{aligned}
		\bfM(\bfx)|_{ij} = &\ \int_{\Ga_h[\bfx]}  \!\! \phi_i[\bfx] \, \phi_j[\bfx]  ,\\
		\bfA(\bfx)|_{ij} = &\ \int_{\Ga_h[\bfx]}  \!\!\!\! \nb_{\Ga_h[\bfx]} \phi_i[\bfx] \cdot \nb_{\Ga_h[\bfx]} \phi_j[\bfx]  .
\end{aligned}
\end{align}
We introduce the vector $\bff(\bfx,\bfu)\in \R^{4\dof}$ with entries
\begin{equation*}
	\begin{aligned}
		\bff(\bfx,\bfu)|_{j+(\ell-1)\dof} = &\int_{\Ga_h[\bfx]} \!\!\! |\nabla_{\Ga_h[\bfx]}n_h|^2 u_h  \cdot  \phi_j[\bfx] ,
	\end{aligned}	
\end{equation*}
for $j = 1,\dotsc,\dof$ and $\ell = 1,2,3,4$.


The formulation of the semi-discretized coupled system for mean curvature flow \eqref{eq:MCF discrete} then leads to the matrix--vector formulation of \cite[equation~(3.4)--(3.5)]{MCF}: writing $\bfM(\bfx)$ and $\bfA(\bfx)$ in short for $\bfM(\bfx)\otimes I_4$ and $\bfA(\bfx)\otimes I_4$ with the 4-dimensional identity matrix $I_4$, we have with $\bfu=(\bfn;\bfH)$
\begin{subequations}
\label{eq:MCF matrix-vector}
	\begin{align} 
		\bfM(\bfx) {\bf\dot\bfu} + \bfA(\bfx) \bfu = &\  \bff(\bfx,\bfu) , \\
		\bfv = &\ - \bfH \bullet \bfn , \\
		{\bf\dot\bfx} = &\ \bfv .
	\end{align}
\end{subequations} 

\section{Semi-discrete methods with area decay}
\label{section:semi-discrete MCF with area decay}

\subsection{Area-decreasing method without tangential velocity}

 The \emph{area-decreasing} evolving FEM for the coupled system \eqref{eq:MCF energy decay} reads as follows. 
 Find nodal vectors $\bfx(t) \in \R^{3\dof}$ determining the surface $\Ga_h[\bfx(t)]$ and their velocities $\bfv(t)\in \R^{3\dof}$ and the nodal vector $\bfu(t)=(\bfn(t);\bfH(t))\in \R^{4\dof} $ of a finite element function $u_h(t)=u_h(\cdot,t)=(n_h(t);H_h(t))\in S_h[\bfx(t)]^4$ 
satisfying the following weak formulation: for all $t\in[0,T]$ and for all finite element test functions  ${\varphi_h}\in S_h[\bfx(t)]^4$ we require (omitting the argument $t$)
\begin{equation}
	\label{eq:MCF discrete area-decreasing}
	\begin{aligned}
		\int_{\Ga_h[\bfx]} \!\!\! \mat_h u_h\cdot \varphi_h + \int_{\Ga_h[\bfx]} \!\!\!\! \nabla_{\Ga_h[\bfx]} u_h \cdot \nabla_{\Ga_h[\bfx]} \varphi_h  = &\ \int_{\Ga_h[\bfx]} \!\!\! 
        |\nabla_{\Gamma_h[\bfx]} n_h|^2 u_h  \cdot \varphi_h  
        \\[2mm]
        \bfv = &\ -  \bfH \bullet \bfn 
        \\
        {\bf\dot\bfx} = &\ \big( 1 + \lambda_h \big)\bfv ,
	\end{aligned}
\end{equation}
where 
$\lambda_h(t) \in \R$ is a Lagrange multiplier which is to ensure that the area (energy) of the computed surface, i.e., 
$$
    \area_h[\bfx(t)] := \area(\Ga_h[\bfx(t)]) = \int_{\Ga_h[\bfx(t)]} 1 ,
$$
decreases along the semi-discrete flow analogously to the continuous flow, see \eqref{eq:MCF energy decay}: 
\begin{equation}
\label{eq:discrete MCF energy decay - a}
	\diff \area_h[\bfx] = - \| H_h \|_{L^2(\Ga_h[\bfx])}^2 \leq 0.
\end{equation}
With the finite element functions $x_h$ and $v_h$ on $\Gamma_h[\bfx]$ with nodal vectors $\bfx$ and $\bfv$, respectively, we obtain
\begin{equation}
\label{eq:discrete MCF energy decay - b}
	\begin{aligned}
		\diff \int_{\Ga_h[\bfx]} 1 
        & = \int_{\Ga_h[\bfx]} \nabla_{\Ga_h[\bfx]} \cdot {\mat_h x_h} 
        = (1+\lambda_h) \int_{\Ga_h[\bfx]}\nabla_{\Ga_h[\bfx]} \cdot v_h ,
	\end{aligned}
\end{equation}
because of the Leibniz formula in the first equality and  because the third relation in \eqref{eq:MCF discrete area-decreasing} and the transport property  $\mat_h \phi_i[\bfx]=0$ of the basis functions imply
$
{\mat_h x_h} 
= (1+\lambda_h) v_h,
$
which yields the second equality.
Therefore, the combination of \eqref{eq:discrete MCF energy decay - a} and \eqref{eq:discrete MCF energy decay - b} yields the following explicit formula for the semi-discrete Lagrange multiplier:
\begin{equation}
\label{eq:def lambda_h}
	\lambda_h(t) = -1 - \frac{\| H_h \|_{L^2(\Ga_h[\bfx])}^2}{\int_{\Ga_h[\bfx]}\nabla_{\Ga_h[\bfx]} \cdot v_h} ,
\end{equation}
provided that the denominator is nonzero.
Here, $\| H_h \|_{L^2(\Ga_h[\bfx])}^2 = \bfH^{\rm T} \bfM(\bfx) \bfH$ with the mass matrix $\bfM(\bfx)$ of \eqref{def-matrix-M-A}.

The matrix--vector formulation of the area-decreasing semi-discretization \eqref{eq:MCF discrete area-decreasing} reads:
with $\bfu=(\bfn;\bfH)$,
\begin{subequations}
	\label{eq:MCF matrix-vector area-decreasing}
	\begin{align} 
        \bfM(\bfx) {\bf\dot\bfu} + \bfA(\bfx) \bfu = &\  \bff(\bfx,\bfu) , \\
        \bfv = &\ - \bfH \bullet \bfn , \\
        {\bf\dot\bfx} = &\ (1 + \lambda_h) \bfv , \\
        \label{eq:MCF matrix-vector area-decreasing - a}
        \frac{\d}{\d t} \area_h[\bfx] = &\ -  \bfH^{\rm T} \bfM(\bfx) \bfH .
	\end{align}
\end{subequations}
Note that \eqref{eq:MCF matrix-vector area-decreasing - a} is equivalent to the explicit expression of $\lambda_h$ in \eqref{eq:def lambda_h}. It is formulated to keep the semi-discrete scheme analogous to the fully discrete scheme that will be presented in the next section.

\subsection{Area-decreasing method with artificial tangential motion}

The area-decreasing evolving FEM with the MDR tangential motion introduced in \cite{Hu-Li-2022} for the coupled system \eqref{KLL2-lambda} reads as follows. Find a nodal vector $\bfx(t) \in \R^{3\dof}$ which determines the approximate surface $\Gamma_h[\bfx(t)]$, as well as nodal vectors $\bfv(t)\in \R^{3\dof}$, $\bfu(t)=(\bfn(t);\bfH(t))\in \R^{4\dof}$ and $\boldsymbol{\mu}(t)\in\R^{\dof}$ such that the corresponding finite element functions $v_h(t)=v_h(\cdot,t)\in S_h[\bfx(t)]^3$, $u_h(t)=(n_h(t);H_h(t))\in S_h[\bfx(t)]^4$ and $\mu_h(t)\in S_h[\bfx(t)]$ satisfy the following weak formulation: for all $t\in [0,T]$ and for all $\varphi_h\in S_h[\bfx(t)]^4$ and $\psi_h\in S_h[\bfx(t)]^3$, $\chi_h\in S_h[\bfx(t)]$,
\begin{subequations}
\label{eq:MCF-MDR}
	\begin{align}
		\int_{\Ga_h[\bfx]} {\mat_h u_h}\cdot {\varphi_h} + \int_{\Ga_h[\bfx]} \!\!\!\! \nabla_{\Gamma_h[\bfx]} u_h \cdot \nabla_{\Gamma_h[\bfx]} {\varphi_h} = &\ \int_{\Ga_h[\bfx]} \!\!\! |\nabla_{\Gamma_h[\bfx]} n_h|^2 u_h \cdot {\varphi_h} \notag\\
        &\
        + \int_{\Ga_h[\bfx]} \!\!\! [(v_h\cdot \nabla_{\Gamma_h[\bfx]}) u_h] \cdot {\varphi_h} , \label{eq:MCF-MDR-eq2}\\[2mm]
        \int_{\Ga_h[\bfx]} \nabla_{\Gamma_h[\bfx]} v_h\cdot\nabla_{\Gamma_h[\bfx]}\psi_h = &\ -\int_{\Ga_h[\bfx]} \mu_h n_h\cdot \psi_h , \\
        \int_{\Ga_h[\bfx]} v_h\cdot n_h\, \chi_h = &\ -\int_{\Ga_h[\bfx]} H_h {\chi_h}  , \label{eq:MCF-MDR-eq3}\\[5mm]
		{\bf\dot\bfx} = &\ \big( 1 + \lambda_h \big) \bfv, \\
        \diff \area_h[\bfx(t)] = &\ - \| H_h \|_{L^2(\Ga_h[\bfx])}^2 ,
	\end{align}
\end{subequations}
where again $\lambda_h(t) \in \R$ is a Lagrange multiplier which ensures that the  area of the computed surface decreases as time grows. According to the discussions in \eqref{eq:discrete MCF energy decay - a}--\eqref{eq:def lambda_h}, $\lambda_h$ is given by the explicit formula \eqref{eq:def lambda_h}.

The matrix--vector formulation of this area-decreasing evolving FEM with  MDR tangential motion reads: with $\bfu=(\bfn;\bfH)$,
\begin{subequations}
	\label{eq:MCF-matrix-vector-MDR}
	\begin{align} 
        \bfM(\bfx) {\bf\dot\bfu} + \bfA(\bfx) \bfu = &\  \bff(\bfx,\bfu,\bfv) , \\[5pt]
        \begin{pmatrix}
            {\bf A}(\bfx) &  {\bf B}(\bfx,\bfn)^{\rm T}
            \\
            {\bf B}(\bfx,\bfn)  & \bfzero
        \end{pmatrix}
        \begin{pmatrix}
            \bfv \\ \boldsymbol{\mu}
        \end{pmatrix}
        =&\
        \begin{pmatrix}
            \bfzero \\ -  {\bf M}(\bfx)\bfH
        \end{pmatrix},
        \\[5pt]
         {\bf\dot\bfx} = &\ (1 + \lambda_h) \bfv , \\[-2mm]
        \frac{\d}{\d t} \area_h[\bfx] = &\ -  \bfH^{\rm T} \bfM(\bfx) \bfH ,
	\end{align}
\end{subequations}
where ${\bf f}({\bf x},{\bf u},{\bf v})$ is the unique nodal vector in $\R^{4\dof}$ satisfying  
$$
{\bf f}({\bf x},{\bf u},{\bf v})^{\rm T} \boldsymbol{\varphi}
= \int_{\Ga_h[\bfx]} \!\!\! |\nabla_{\Gamma_h[\bfx]} n_h|^2 u_h \cdot {\varphi_h} 
  + \int_{\Ga_h[\bfx]} \!\!\! [(v_h\cdot \nabla_{\Gamma_h[\bfx]}) u_h] \cdot {\varphi_h} 
$$
for all finite element functions $\varphi_h\in S_h[{\bf x}]^4$ with nodal vector $\boldsymbol{\varphi}$, and ${\bf B}(\bfx,\bfn)$ is the unique $\dof\times3\dof$ matrix satisfying the following condition:  
$$
\boldsymbol{\chi}^{\rm T} {\bf B}(\bfx,\bfn)\boldsymbol{\psi}
= \int_{\Ga_h[\bfx]} \psi_h\cdot n_h \chi_h
\quad\mbox{for all}\,\, \boldsymbol{\psi} \in\R^{3\dof}\,\,\mbox{and}\,\,\, \boldsymbol{\chi} \in \R^{\dof}, 
$$
for all finite element functions $\psi_h\in S_h[\bfx]^3$ and $\chi_h\in S_h[\bfx]$ with nodal vectors $\boldsymbol{\psi} $ and $\boldsymbol{\chi}$, respectively.

In the next section, we introduce the corresponding fully discrete FEMs  with Lagrange multipliers. Error bounds for the semi-discrete FEMs
in~\eqref{eq:MCF matrix-vector area-decreasing} and~\eqref{eq:MCF-matrix-vector-MDR}
(without and with tangential motion, respectively) follow directly from the fully discrete
error bounds proved in Theorems~\ref{theorem-fd-1} and~\ref{theorem-fd-2} by taking the
limit \(\tau \to 0\). Therefore, we do not state the semi-discrete error bounds separately.

\section{Fully discrete methods with area decay}

For the time discretization of the system of ordinary differential equations \eqref{eq:MCF matrix-vector area-decreasing} we use a multistep approach: a $q$-step linearly implicit backward difference formula (BDF method), with $2 \le q \leq 5$, for the (stiff) geometric equation, and a $q$-step implicit Adams method for the (nonstiff) differential equation for the positions. 

\subsection{Area-decreasing method without tangential velocity}

The $q$-step backward difference formula (BDF) method is determined by its coefficients $\delta_j$, given by $\delta(\zeta)=\sum_{j=0}^q \delta_j \zeta^j=\sum_{\ell=1}^q \frac{1}{\ell}(1-\zeta)^\ell$. In its linearly implicit version, the BDF method uses extrapolation coefficients $\gamma_j$ defined by $\gamma(\zeta) = \sum_{j=0}^{q-1} \gamma_j \zeta^j = (1 - (1-\zeta)^q)/\zeta$. 
The  BDF method is known to be zero-stable and $A(\alpha)$-stable with $\alpha>0$ for $q\leq 6$, see \cite[Chapter~V]{HairerWannerII}, and for $q\le 5$, to converge with order $q$ for linear parabolic equations on moving surfaces \cite{LubichMansourVenkataraman_bdsurf}.
This order is retained by the linearly implicit variant for quasilinear parabolic equations using the above coefficients $\gamma_j$; cf.~\cite{AkrivisLubich_quasilinBDF,AkrivisLiLubich_quasilinBDF}.

For a fixed step size $\tau>0$, we define $t_n = n \tau$ for $n=0,1,2,\dots$, and denote by $\widetilde \bfx^n$ the $q$-step extrapolated values defined by the following formula: 
\begin{equation}
\label{eq:extrapolation def}
    \widetilde \bfx^n := 
    \left\{\begin{aligned}
    &\sum_{j=0}^{q-1} \gamma_j \bfx^{n-1-j} &&\mbox{for}\,\,\,	n \geq q , \\
    &\bfx^n &&\mbox{for}\,\,\, 0\le n \le q-1 .
    \end{aligned}\right.
\end{equation}
We use the analogous definition for $\wt \bfu^n=(\wt \bfn^n;\wt \bfH^n)$. Moreover, we denote by ${\bf\dot\bfu}^n$ the $q$-step backward finite difference defined by 
\begin{align}\label{eq:BDF def}
{\bf\dot\bfu}^n := \frac{1}{\tau} \sum_{j=0}^q \delta_j \bfu^{n-j} .
\end{align}
We further define the weighted average of the velocity values that appears in the $q$-step implicit Adams method,
\begin{align}\label{def-v-beta}
{\bar\bfv}^n
= \sum_{j=0}^{q}\beta_j\,\mathbf{v}^{\,n-j},
\end{align}
with the weights 
\(
\beta_j=\frac{1}{\tau}\int_{t_{n-1}}^{t_n} L_j^{n}(t)\,\mathrm{d}t,
\)
where \(L_j^{n}(t)\) are the Lagrange basis polynomials of degree \(q\) associated with the
nodes \(\{t_n,t_{n-1},\ldots,t_{n-q}\}\), i.e., 
$L_j^{n}(t_{n-i})=\delta_{ij}$ for $i,j=0,\ldots,q$.
The coefficients \( (\beta_j)_{j=0}^{q}\) define the \(q\)-step implicit Adams method and can be precomputed and stored for use in the implementation; see Table~\ref{tab:beta-coeffs}. It is well known that the \(q\)-step implicit Adams method has convergence order \(q+1\) when applied to nonstiff ordinary differential equations; see \cite[Chapter~III]{HNW}.

\begin{table}[htbp]
\centering
\caption{The coefficients \(\beta_j\) in the $q$-step implicit Adams method.}
\label{tab:beta-coeffs}
\normalsize
\begin{tabular}{ccccccc}
\hline
\(q\) & \(\beta_0\) & \(\beta_1\) & \(\beta_2\) & \(\beta_3\) & \(\beta_4\) & \(\beta_5\) \\[3pt]
\hline
0 & 1 &  &  &  &  &  \\[3pt]
1 & \(\tfrac12\) & \(\tfrac12\) &  &  &  &  \\[3pt]
2 & \(\tfrac{5}{12}\) & \(\tfrac{8}{12}\) & \(-\tfrac{1}{12}\) &  &  &  \\[3pt]
3 & \(\tfrac{9}{24}\) & \(\tfrac{19}{24}\) & \(-\tfrac{5}{24}\) & \(\tfrac{1}{24}\) &  &  \\[3pt]
4 & \(\tfrac{251}{720}\) & \(\tfrac{646}{720}\) & \(-\tfrac{264}{720}\) & \(\tfrac{106}{720}\) & \(-\tfrac{19}{720}\) &  \\[3pt]
5 & \(\tfrac{475}{1440}\) & \(\tfrac{1427}{1440}\) & \(-\tfrac{798}{1440}\) & \(\tfrac{482}{1440}\) & \(-\tfrac{173}{1440}\) & \(\tfrac{27}{1440}\) \\[3pt]
\hline
\end{tabular}
\end{table}

For $n\ge q$, we determine the approximations $\bfx^n$ to $\bfx^*(t_n)$, $\bfv^n$ to $\bfv^*(t_n)$, $\bfu^n$ to $\bfu^*(t_n)$ by the fully discrete linearly implicit method
\begin{subequations}
	\label{eq:multistep area-decreasing}
	\begin{align} 
        \label{eq:multistep area-decreasing - u}
        \bfM(\wtx^n) {\bf\dot\bfu}^n + \bfA(\wtx^n) \bfu^n = &\  \bff(\wtx^n,\wtu^n) , \\
        \label{eq:multistep area-decreasing - v}
        \bfv^n = &\ - \bfH^n \bullet {\bfn}^n , \\
        \label{eq:multistep area-decreasing - x}
        \frac{1}{\tau} \big( \bfx^n - \bfx^{n-1} \big) = &\ (1 + \lambda^n) {\bar\bfv}^n, \\
        \label{eq:multistep area-decreasing - area}
         \area_h[\bfx^n] - \area_h[\bfx^{n-1}]
        = &\ - \int_{t_{n-1}}^{t_n} p^n(t)^2 \,\d t , 
	\end{align}
\end{subequations}
where $\lambda^n \in \R$ is a Lagrange multiplier which enforces \eqref{eq:multistep area-decreasing - area}, $\bar\bfv^n$ is defined by \eqref{def-v-beta}, and $p^n(t)$ is the interpolation polynomial through the points $(t_{n-j},\|H_h^{n-j}\|_{L^2(\Ga_h[\wtx^{n-j}])})$ for $0\le j \le q$, where  $\|H_h^{n-j}\|_{L^2(\Ga_h[\wtx^{n-j}])} = \linebreak \bigl((\bfH^{n-j})^T\bfM(\wtx^{n-j})\bfH^{n-j}\bigr)^{1/2}$. The surface area thus decreases with a rate consistent with the approximate mean curvature given by the numerical scheme; cf.~\eqref{area-decay}.

The starting values $\bfx^i$ and $\bfu^i$ ($i=0,\dotsc,q-1$) are assumed to be given. They can be precomputed using either a lower order method with smaller step sizes or an implicit Runge--Kutta method.

Due to the linearly implicit time discretisation, each time step requires the solution of a \emph{decoupled} system, solved sequentially: first solving a linear system for the geometry $\bfu^n=(\bfn^n;\bfH^n)$, then determining the velocity $\bfv^n$, then a simplified Newton iteration for $\bflambda^n$, which finally determines $\bfx^n$. For practical details, we refer to Section~\ref{section:implementation}.

The following result shows the existence of a Lagrange multiplier $\lambda^n$ and its uniqueness in a small interval around $0$, together with the convergence of the simplified Newton method under conditions that are substantially weaker than the strong regularity conditions required for the optimal-order error bounds in Theorem~\ref{theorem-fd-1} below. We need some notation.
Let $\bar\bfx^n = \bfx^{n-1} + \tau \bar \bfv^{n}$ and denote the defect in \eqref{eq:multistep area-decreasing - area} as
\begin{equation}\label{delta-n}
\delta^n := \left| \frac1\tau  \bigl(\area_h[\bar\bfx^n] - \area_h[\bfx^{n-1}]\bigr) +
        \frac1\tau \int_{t_{n-1}}^{t_n} p^n(t)^2 \,\d t \right|.
\end{equation}
We assume that 
\begin{equation}\label{sigma-n}
\sigma^n := \Big| \int_{\Gamma_h[\bar \bfx^n]} \nabla_{\Gamma_h[\bar \bfx^n]}\cdot \bar v_h^n \Big| >0,
\end{equation}
where $\bar v_h^n \in S_h[\bar \bfx^n]^3$ is the finite element function 
on $\Gamma_h[\bar \bfx^n]$ with the nodal vector $\bar \bfv^n$ defined in \eqref{def-v-beta}. In view of \eqref{area-decay}, this assumption is fulfilled as long as the numerical positions and velocities remain $H^1$-close to the exact positions and velocities.


With an absolute constant $c_2$ (with $c_2\le 4$), we further define 
\begin{equation}\label{eta-n}
\eta^n := 3\biggl(\frac{c_2\, \| \nabla_{\Gamma_h[\bar\bfx^n]}\bar v_h^n \|_{L^2(\Gamma_h[\bar \bfx^n])^{3\times 3}}}{\sigma^n} \biggr)^2.
\end{equation}

\begin{proposition}[Existence of the Lagrange multiplier] \label{prop:lambda} In the above situation,
assume that the step size $\tau$ and 
the defect $\delta^n$ are small enough that 
\begin{align}\label{gamma}
&\tau\delta^n\eta^n\le \tfrac14, \text{ \ so that \ $\tau\delta^n\eta^n=\gamma^n(1-\gamma^n)$ for some $\gamma^n\in(0,\tfrac12]$,}
\\ 
\label{rho}
\text{and }\ &\rho^n := \frac{\delta^n}{\sigma^n (1-\gamma^n)}
\ \text{ be such that }\
 \tau\rho^n \| \nabla_{\Gamma_h[\bar\bfx^n]}\bar v_h^n \|_{L^\infty(\Gamma_h[\bar \bfx^n])} \le \half. 
\end{align}
Then, there exists a unique Lagrange multiplier $\lambda^n$ with $|\lambda^n|\le \rho^n$ such that 
\eqref{eq:multistep area-decreasing - x}--\eqref{eq:multistep area-decreasing - area} are satisfied. In particular, this implies that the area decreases. 
Moreover, a simplified Newton method for $\lambda^n$ with initial iterate~$0$ converges linearly to $\lambda^n$ with convergence factor $\gamma^n$.
\end{proposition}

\begin{proof}
    In the proof it is convenient to omit the ubiquitous superscript $n$. We aim to find a zero of the function
    $$
    f(\lambda) = \frac1\tau  \bigl(\area_h[\bar\bfx+\tau\lambda \bar\bfv] - \area_h[\bfx^{n-1}]\bigr) +
        \frac1\tau \int_{t_{n-1}}^{t_n} p(t)^2 \,\d t
    $$
    and note that $|f(0)|=\delta$ and $|f'(0)|=\sigma$, see \eqref{delta-n} and \eqref{sigma-n}.
    We consider the simplified Newton iteration
    $$
    \lambda_{j+1} = \lambda_j - \frac{f(\lambda_j)}{f'(0)}, \qquad j\ge 0,
    $$
    with initial value $\lambda_0=0$. This is a fixed-point iteration for the map $\lambda\mapsto \varphi(\lambda):=\lambda - f'(0)^{-1}f(\lambda)$. We show that $\varphi$ is a contraction in the closed ball of radius 
    $$
    \rho=\frac{|f'(0)^{-1}f(0)|}{1-\gamma} = \frac{\delta} {\sigma (1-\gamma)}.
    $$ 
    For this we prove that the derivative satisfies
    $$
    |\varphi'(\lambda)| =|f'(0)^{-1}(f'(0)-f'(\lambda))| \le \gamma <1 
    \quad\text{for }\ |\lambda|\le \rho.
    $$
    For ease of notation, we abbreviate in the following for a fixed $\lambda$
    $$
    \Gamma_h^\theta = \Gamma_h[\bar \bfx+\theta\tau\lambda \bar\bfv]
    \quad\text{and}\quad \bar v_h^\theta = \bar v_h [\bar \bfx+\theta\tau\lambda \bar\bfv]
    \quad\text{for}\quad 0\le \theta \le 1,
    $$
    where $\bar v_h[\bar \bfx+\theta\tau\lambda \bar\bfv]\in S_h[\bar \bfx+\theta\tau\lambda \bar\bfv]$ is the finite element function on $\Gamma_h[\bar \bfx+\theta\tau\lambda \bar\bfv]$ with the nodal vector $\bar \bfv$.
    We have
    \begin{align*}
    f'(\lambda) -f'(0) &=  \int_{\Gamma_h^1} \nabla_{\Gamma_h^1}\cdot \bar v_h^1 -  \int_{\Gamma_h^0} \nabla_{\Gamma_h^0}\cdot \bar v_h^0
    \\
    &= 
     \int_0^1 \frac{\d}{\d \theta}\int_{\Gamma_h^\theta} \nabla_{\Gamma_h^\theta}\cdot \bar v_h^\theta \ \d \theta
    \\
    &=    \int_0^1 \int_{\Gamma_h^\theta} \biggl(\tau\lambda\bigl(\nabla_{\Gamma_h^\theta}\cdot \bar v_h^\theta \bigr)^2    + \mat_\theta \bigl(\nabla_{\Gamma_h^\theta}\cdot \bar v_h^\theta \bigr) \biggr)
    \,d\theta,
    \end{align*}
    where the last equality is the Leibniz formula.
     Using the formula for $\mat \nabla_\Gamma v$ in \cite[Lemma 2.6]{DziukKronerMuller}, we find
     $$
     |\mat_\theta \bigl(\nabla_{\Gamma_h^\theta}\cdot \bar v_h^\theta \bigr)| \le 2\tau|\lambda|\, |\nabla_{\Gamma_h^\theta} \bar v_h^\theta|_F^2,
     $$
     where $|\cdot|_F$ is the Frobenius norm of a matrix. This yields the bound
    \begin{align*}
    |f'(\lambda) -f'(0)| &\le  3\tau|\lambda| \int_0^1 \| 
    \nabla_{\Gamma_h^\theta} \bar v_h^\theta \|_{L^2(\Gamma_h^\theta)^{3\times 3}}^2\, \d \theta.
    \end{align*}
    Lemma 4.3 in \cite{KLLP2017} yields that the condition
    $$
   \tau |\lambda| \| \nabla_{\Gamma_h^0}\bar v_h \|_{L^\infty(\Gamma_h^0)} \le \half, 
    $$
    which is implied by \eqref{rho} for $|\lambda|\le \rho$, guarantees norm equivalence on all intermediate surfaces, i.e.,
    $$
    \| 
    \nabla_{\Gamma_h^\theta} \bar v_h^\theta \|_{L^2(\Gamma_h^\theta)} \le c_2 
    \| 
    \nabla_{\Gamma_h^0} \bar v_h^0 \|_{L^2(\Gamma_h^0)} =:\beta .
    $$
    So we obtain 
    $$
    |f'(\lambda) -f'(0)|  \le  3 \tau |\lambda|  \, \beta^2.
    $$
    Using \eqref{gamma} and \eqref{rho}, we therefore have for $|\lambda|\le \rho$,
    \begin{align*}
    &|\varphi'(\lambda)|=|f'(0)^{-1}(f'(0)-f'(\lambda))| \le \frac 1\sigma \, 3\tau\,\frac{\delta} {\sigma (1-\gamma)} \,\beta^2
    =\frac{\tau\delta\eta}{1-\gamma} =\gamma \le \half
    \\
    \text{and } \quad &|\varphi(\lambda)| \le |\varphi(\lambda)-\varphi(0)| + |\varphi(0)| \le \gamma\rho + \frac\delta\sigma =\rho.
    \end{align*}
    This shows that $\varphi:[-\rho,\rho]\to [-\rho,\rho]$ is a contraction. The Banach fixed-point theorem then yields the result.
    \qed
\end{proof}
\ecl


The following result states that the above area-decreasing fully discrete method satisfies optimal-order error bounds in the $H^1$-norm of the same type as the original full discretization in \cite{MCF}, which is not guaranteed to decrease the area under weak regularity assumptions.

\begin{theorem}[Error of the area-decreasing fully discrete method without tangential motion] \label{theorem-fd-1} 
    Assume that the mean curvature flow problem \eqref{KLL-lambda} admits a smooth solution consisting of $X:\Gamma^0\times[0,T]\rightarrow\mathbb{R}^3$, $(v,n,H):\Gamma(t)\rightarrow\mathbb{R}^3\times\mathbb{R}^3\times \mathbb{R}$ and $\lambda=0$ for $t\in[0,T]$, with $X(\cdot,t):\Gamma^0\rightarrow\Gamma(t)$ being a smooth diffeomorphism for $t\in[0,T]$.

    For any constant $c_0$, there exists a positive constant $h_0$ such that the fully discrete area-decreasing algorithm \eqref{eq:multistep area-decreasing}, using finite elements of polynomial degree~${k \geq 2}$ and $q$-step BDF and $q$-
    step implicit Adams methods with $2\le q\le 5$, admits a unique solution on the time interval $[0,T]$ when $\tau\leq c_0 h$ 
    and $h\le h_0$ (with a locally unique solution $\lambda^n$). 
    
  In the following, let the superscript $L$ denote the lift of a function from the triangulated surface $\Ga_h[\bfx^n]$ to the exact surface $\Ga(t_n)=\Ga(X(\cdot,t_n))$, as defined in \cite{MCF}. If the starting values are sufficiently accurate in the $H^1$-norm, i.e., for $j=1,\ldots,q-1$, 
\begin{align*}
\|(x_h^j)^L - {\rm id}_{\Gamma(t_j)} \|_{H^1(\Gamma(t_j))}
&\le c \, (h^k+\tau^q), \\
\|(n_h^j)^L - n(\cdot,t_j) \|_{H^1(\Gamma(t_j))}
+ \|(H_h^j)^L - H(\cdot,t_j) \|_{H^1(\Gamma(t_j))} 
&\le c \, (h^k+\tau^q), \\
\tau^{1/2}
\Big\|
\frac{(x_h^j)^L-{\rm id}_{\Gamma(t_j)}}{\tau}
\Big\|_{H^1(\Gamma(t_j))}
&\le c \, (h^k+\tau^q),
\end{align*}
then the following optimal-order error bounds hold for $t_n\in[0,T]$ with $n\geq q${\rm:} 
    \begin{subequations}\label{THM1-fd-error}
    \begin{align}
       \|(x_h^n)^L - \mathrm{id}_{\Gamma(t_n)}\|_{H^1(\Ga(t_n))^3} &\leq c \, (h^k+\tau^q), \\
        \|(v_h^n)^L - v(\cdot,t_n)\|_{H^1(\Ga(t_n))^3} 
        & \leq c \, (h^k+\tau^q), \\
        \|(n_h^n)^L - n(\cdot,t_n)\|_{H^1(\Ga(t_n))^3} 
        & \leq c \, (h^k+\tau^q), \\
        \|(H_h^n)^L - H(\cdot,t_n)\|_{H^1(\Ga(t_n))} 
        & \leq c \, (h^k+\tau^q) , \\
        |\lambda^n| & \leq c \, (h^k+\tau^q) ,
    \end{align}
    \end{subequations}
    where the constants $c$ are independent of $h$ and $\tau$ and $n$ with $n\tau\le T$, but depend on bounds of derivatives of the solution and on the length $T$ of the time interval. Moreover, the
simplified Newton method for $\lambda^n$ with initial iterate 0 converges linearly 
with a convergence factor $O(\tau(h^k+\tau^q))$.
\end{theorem}

\subsection{Area-decreasing fully discrete FEM with tangential velocity}

Analogous to \eqref{eq:multistep area-decreasing}, while using the matrix-vector formulation in \eqref{eq:MCF-matrix-vector-MDR}, the fully discrete method with tangential velocity can be described as follows. For $n\ge q$, we determine the approximations $\bfx^n$ to $\bfx^*(t_n)$, $\bfv^n$ to $\bfv^*(t_n)$, $\boldsymbol{\mu}^n$ to $\boldsymbol{\mu}^*(t_n)$, $\bfu^n$ to $\bfu^*(t_n)$ by the fully discrete system of linear equations:
\begin{subequations}
	\label{eq:multistep-MDR}
	\begin{align} 
        \label{eq:multistep-MDR - u}
        \bfM(\wtx^n) {\bf\dot\bfu}^n + \bfA(\wtx^n) \bfu^n = &\  \bff(\wtx^n,\wtu^n,\wtv^n) , \\[5pt]
        \label{eq:multistep-MDR - v2}
        \begin{pmatrix}
            {\bf A}(\wt\bfx^n) &  {\bf B}(\wt\bfx^n,\bfn^n)^{\rm T}
            \\
            {\bf B}(\wt\bfx^n,\bfn^n)  & \bfzero
        \end{pmatrix}
        \begin{pmatrix}
            \bfv^n \\ \boldsymbol{\mu}^n
        \end{pmatrix}
        =&\
        \begin{pmatrix}
            \bfzero \\ -  {\bf M}(\wt\bfx^n)\bfH^n
        \end{pmatrix}, \\[5pt]
        \label{eq:multistep-MDR - x}
        \frac{1}{\tau} \big( \bfx^n - \bfx^{n-1} \big) = &\ (1 + \lambda^n) {\bar\bfv}^n , \\
        \label{eq:multistep-MDR - area}
         \area_h[\bfx^n] - \area_h[\bfx^{n-1}] 
        = &\ - \int_{t_{n-1}}^{t_n} p^n(t)^2 \, \d t ,
	\end{align}
\end{subequations} 
where $\widetilde{\bf v}^n=\sum_{j=0}^{q-1}\gamma_j {\bf v}^{n-1-j}$ and ${\bar\bfv}^n:=\sum_{j=0}^{q} \beta_j \bfv^{n-j} $ and $p^n(t)$ is again the interpolation polynomial through the points 
$(t_{n-j},\|\bfH^{n-j}\|_{\bfM(\wtx^{n-j})})$ for $0\le j \le q-1$. The solvability of the saddle point system \eqref{eq:multistep-MDR - v2} can be shown through error analysis for sufficiently small $h$, as shown in \cite{Hu-Li-2022}.

\begin{theorem}[Error of the area-decreasing method with tangential motion] \label{theorem-fd-2} 
 Assume that the mean curvature flow problem \eqref{KLL2-lambda} (with the MDR tangential motion) admits a smooth solution $X:\Gamma^0\times[0,T]\rightarrow\mathbb{R}^3$ and $(v,n,H,\mu):\Gamma(t)\rightarrow\mathbb{R}^3\times\mathbb{R}^3\times \mathbb{R}\times \mathbb{R}$ for $t\in[0,T]$, with $X(\cdot,t):\Gamma^0\rightarrow\Gamma(t)$ being a smooth diffeomorphism for $t\in[0,T]$. Then the result of Theorem~\ref{theorem-fd-1} holds verbatim also for method \eqref{eq:multistep-MDR} (the error bound for $\mu$ is not included).

\end{theorem}

\section{Proof of Theorem~\ref{theorem-fd-1}}
\label{Section:Proof-fd-1}

The proof is a modification of the proof of the fully discrete error bounds in \cite{MCF}. The only new aspect is the effect on the stability of the discrete Lagrange multipliers $\lambda^n$  of the numerical method, which is shown to be mild. We focus on this new aspect and recall results and techniques from the stability and consistency analysis of \cite{MCF} as far as needed.

\subsection{Preparations}
\label{subsec:prep}

\subsubsection{Notation of norms of finite element functions}
For the simplicity of notation, we identify a nodal vector ${\bf w}$ with the corresponding finite element function on a specified triangulated surface, such as $\Gamma_h[\bfx]$ or $\Gamma_h[\bfx^*]$, whenever we consider the norm 
of the finite element function on the triangulated surface. For example, for $s\in\{0,1\}$ and $1\le p\le\infty$, we denote by $\| {\bf w} \|_{H^s(\Gamma_h[\bfx])} $ and $\| {\bf w} \|_{W^{s,p}(\Gamma_h[\bfx])} $ the Sobolev norms of the finite element function on $\Gamma_h[\bfx]$ having nodal vector ${\bf w}$. In particular,  
\begin{subequations}\label{norm-w}
\begin{align}
\| {\bf w} \|_{L^2(\Gamma_h[\bfx])} 
&= \sqrt{{\bf w}^{\rm T}\bfM(\bfx){\bf w}} ,\\
\| \nabla_{\Gamma_h[\bfx]}{\bf w} \|_{L^2(\Gamma_h[\bfx])} 
&= \sqrt{{\bf w}^{\rm T}\bfA(\bfx){\bf w}} ,\\
\| {\bf w} \|_{H^1(\Gamma_h[\bfx])} 
&= \sqrt{{\bf w}^{\rm T}\bigl(\bfM(\bfx)+\bfA(\bfx)\bigr){\bf w}}, 
\end{align}
\end{subequations}
for the matrices $\bfM(\bfx)$ and $\bfA(\bfx)$ defined in \eqref{def-matrix-M-A}.

\subsubsection{Finite element interpolation of position and velocity}
Let $\bfx^*(t)=(x_1^*(t), \cdots, x_\dof^*(t))$ be the image of $\bfx^0$ under the exact flow map $X(\cdot,t):\Gamma^0\rightarrow\Gamma(t)$, i.e., $x_j^*(t)=X(x_j^0,t)$ for $j=1,\dots,\dof$. Thus the nodal vector $\bfx^*(t)$ determines a triangulated surface $\Gamma_h[\bfx^*(t)]$ which interpolates $\Gamma(t)$. The nodal values of $v(\cdot,t)$  at the interpolated nodes $x_j^*(t)$, $j=1,\dots,\dof$, are collected into the nodal vector $\bfv^*(t)$. We note that $\dot\bfx^*=\bfv^*$. We denote $\bfx_*^n={\bfx}^*(t_n)$ and $\bfv_*^n={\bfv}^*(t_n)$.

\subsubsection{Ritz projection of normal vector and mean curvature}
For the solution $u=(n;H) \in H^1(\Ga[X])^4$, we denote by $u_h^* \in S_h[\xs]^4$ the Ritz projection of $u$, i.e., the unique finite element solution of the following weak formulation: 
	\begin{equation}
		\label{eq:Ritzmap}
		\begin{aligned}
			& \int_{\Ga_h[\xs]} \!\!  u_h^* \, \vphi_h + \int_{\Ga_h[\xs]} \!\! \nb_{\Ga_h[\xs]} u^*_h \cdot \nb_{\Ga_h[\xs]} \vphi_h \\
			&=  \int_{\Ga[X]} \!\! u \vphi_h^\ell + \int_{\Ga[X]} \!\! \nb_{\Ga[X]} u \cdot \nb_{\Ga[X]} \vphi_h^\ell 
            \qquad \forall\, \vphi_h \in S_h[\xs]^4. 
		\end{aligned}
	\end{equation}
Such Ritz projections were used in \cite[Definition~6.1]{DziukElliott_L2}, \cite[Definition~6.1]{highorderESFEM}, \cite[Definition~3.6]{ElliottRanner_unified} and played an important role in the error analysis of \cite{MCF}. We recall that (cf.~\cite[Theorem~3.1--3.2]{KovacsPower_quasilinear})
\begin{eqnarray}  
	\label{eq:ritzest}
	\Vert (u^*_h)^\ell - u \Vert_{L^2(\Ga[X])^4} + h \Vert (u^*_h)^\ell - u \Vert_{H^1(\Ga[X])^4} \leq C h^{k+1},  \\
	\Vert (\partial^\bullet_h u^*_h)^\ell - \partial^\bullet u \Vert_{L^2(\Ga[X])^4} + h \Vert (\partial^\bullet_h u^*_h)^\ell - \partial^\bullet u \Vert_{H^1(\Ga[X])^4} \leq C h^{k+1}.  \label{eq:ritzesttime}
\end{eqnarray}
Here $\ell$ is the lift from the interpolated surface to the exact surface, as introduced for linear and higher-order surface approximations in \cite{Dziuk88} and \cite{demlow2009higher}, respectively.
Let $\uls^n={\bfu}^*(t_n)$ be the nodal vector associated with the Ritz projection $u_h^*(\cdot,t_n)$. 

\subsubsection{Equivalences of neighbouring norms}
We have the norm-equivalence relation with constants independent of $h$ and~$\tau$  (when $\tau$ is sufficiently small):
$$
\|\bfw\|_{H^1(\Gamma_h[\bfx_*(t)])}
\simeq \|\bfw\|_{H^1(\Gamma_h[\bfx_*^n])}
\quad\mbox{for}\,\,\, t\in[t_{n-1},t_n] , 
$$
which holds as long as $\|\bfx_*(t)-\bfx_*^n\|_{W^{1,\infty}(\Gamma_h[\bfx_*^n])} \le c\tau\le \frac12$  \cite[Lemma 4.3]{KLLP2017}. For the same reason, we also have the following norm equivalence:
\begin{equation}\label{norm-equiv-n-j}
\|\bfw\|_{H^1(\Gamma_h[\bfx_*^n])}
\simeq \|\bfw\|_{H^1(\Gamma_h[\bfx_*^{n-j}])}
\quad\mbox{for}\,\,\, j=1,\dots,q-1. 
\end{equation}
Analogous norm equivalences hold for the corresponding $L^2$-norms.

\subsection{Error equations}

We can write down the following matrix-vector equations satisfied by the interpolated values $\xls^n={\bfx}^*(t_n)$, $\vls^n={\bfv}^*(t_n)$ and Ritz-projected values $\uls^n= ({\bfn}_*^n;\bfH_*^n)={\bfu}^*(t_n)$ of the exact solution: 
\begin{align*} 
    \bfM(\wtx_*^n) {\bf\dot\bfu}_*^n + \bfA(\wtx_*^n) \bfu_*^n = &\  \bff(\wtx_*^n,\wtu_*^n) + \bfM(\wtx_*^n)\du^n , \\[3pt]
    \bfv_*^n = &\ - \bfH_*^n \bullet {\bfn}_*^n + \dv^n , \\
    \frac{1}{\tau} \big( \bfx_*^n - \bfx_*^{n-1} \big) = &\  \bar\bfv_{*}^n + {\bfd}_{\bfx}^n , 
\end{align*}
where ${\bf\dot\bfu}_*^n = \tau^{-1} \sum_{j=0}^q \delta_j \bfu_*^{n-j}$ and $\wtx_{*}^n=\sum_{j=0}^{q-1} \gamma_j \bfx_*^{n-1-j}$ and $\bar\bfv_{*}^n=\sum_{j=0}^q \beta_j \bfv_*^{n-j}$,  and
where  $\du^n$ and $\dx^n$ correspond to the defects defined in \cite[equation~(10.1)]{MCF}, and $\dv^n$ is the defect as in \cite[equation~(5.14b)]{Willmore}. The following estimates have been shown in \cite{MCF} for $\du^n$, in \cite{Willmore} for $\dv^n$,
and $\dx^n$ is the consistency error of the Adams multistep method:
\begin{align}\label{defects-du-dx-uv}
\| \du^n \|_{L^2(\Gamma_h[\bfx_*^n])} 
+ \| \dx^n \|_{H^1(\Gamma_h[\bfx_*^n])} 
+ \| \dv^n \|_{H^1(\Gamma_h[\bfx_*^n])} \le c \, (\tau^q+h^k) . 
\end{align} 

In \cite{MCF} and \cite{Willmore}, the defect estimates were shown in the $L^2(\Gamma_h[\wtx_*^n])$ and $H^1(\Gamma_h[\wtx_*^n])$ norms. The change to the $L^2(\Gamma_h[\bfx_*^n])$ and $H^1(\Gamma_h[\bfx_*^n])$ norms made here is possible because of the equivalence relation in \eqref{norm-equiv-n-j}. 

We introduce the fully discrete errors at time $t_n = n \tau$:
\begin{equation*}
    \ex^n =  \bfx^n - \xls^n, \quad \ev^n = \bfv^n - \vls^n , \quad \eu^n = \bfu^n - \uls^n .
\end{equation*}
The error equations can then be written as 
\begin{subequations}
\label{eq:error equations - full}
    \begin{align}
        \label{eq:error eq - u - full}
        \bfM(\widetilde \bfx^n){\bf\dot\bfe}_{\bfu}^n+ \bfA(\widetilde \bfx^n)\eu^n 
        = &\ -  \big( \bfM(\widetilde \bfx^n)-\bfM(\widetilde \bfx_*^n) \big) {\bf\dot\bfu}_*^n \nonumber \\[3pt]
        &\ - \big( \bfA(\widetilde \bfx^n)-\bfA(\widetilde \bfx_*^n) \big) \bfu_*^n \nonumber \\[3pt]
        &\ + \big(\bff(\widetilde \bfx^n,\widetilde \bfu^n) - \bff(\widetilde \bfx_*^n,\widetilde \bfu_*^n)\big) -  \bfM(\wtxls^n)  \du^n \\[2pt]
        \label{eq:error eq - v - full}
        \ev^n = &\ -\Big( \bfH^n \bullet \bfn^n - \Hls^n \bullet \nls^n \Big) - \dv^n , \\
        \label{eq:error eq - x - full}
        \frac{1}{\tau} \big( \ex^n - \ex^{n-1} \big) = &\  \evb^n + \lambda^n \bfvb^n  - {\bfd}_{\bfx}^n ,
    \end{align}
\end{subequations}
where $\evb^n=\sum_{j=0}^q \beta_j \ev^{n-j}$.

\subsection{Stability estimates for $\bfe_\bfu^n$ and $\bfe_\bfv^n$}

The proof of the error bounds uses an induction argument. We assume that for some positive integer $m\le T/\tau$ the following estimate holds for all integers $n$ with $1\le n\le m$:
\begin{align}\label{math-ind} 
        & \| \eu^{n-1} \|_{H^1(\Ga_h[\xls^{n-1}])} 
        + \| \ex^{n-1} \|_{H^1(\Ga_h[\xls^{n-1}])}
        + \| \ev^{n-1} \|_{H^1(\Ga_h[\xls^{n-1}])} 
        + |\lambda^{n-1}| \notag\\
        & \le h^{3/2} ,
\end{align}
where $\lambda^i = 0$ for $i=0,\dotsc,q-1$. 
Under the assumptions of Theorem \ref{theorem-fd-1}, for $\tau^q\le ch^2$ and sufficiently small $h$, \eqref{math-ind} holds at least for $m=q$. We shall prove that if this result holds for some $q\le m<T/\tau$ then it also holds for $m+1$. This will prove that \eqref{math-ind} holds for all $q\le m\le T/\tau$.

Note that by applying the finite element inverse estimates (see, e.g., \cite[Theorem~4.5.11]{BrennerScott}), the following estimates for $n\le m$ follows from \eqref{math-ind}:
\begin{align}\label{math-ind-W1inf} 
        & \| \eu^{n-1} \|_{W^{1,\infty}(\Ga_h[\xls^{n-1}])} 
        + \| \ex^{n-1} \|_{W^{1,\infty}(\Ga_h[\xls^{n-1}])}
        + \| \ev^{n-1} \|_{W^{1,\infty}(\Ga_h[\xls^{n-1}])} \notag\\
        & \le c_{1,\infty} h^{1/2} .
\end{align}
Under this condition, stability estimates for \eqref{eq:error eq - u - full} and \eqref{eq:error eq - v - full} for $q\le n\le m$ have been established in \cite[inequality (10.32)]{MCF}: 
\begin{align}\label{stability-eu-2}
\|\eu^n\|_{H^1(\Gamma_h[\bfx_*^n])}^2 
    \le &\ c\sum_{j=0}^{q-1} \| \ex^{n-1-j} \|_{H^1(\Gamma_h[\bfx_*^{n-1-j}])}^2 \notag\\
  &\ +  c\tau\sum_{j=q}^{n-1} \Bigl( \|\eu^j\|_{H^1(\Gamma_h[\bfx_*^{j}])}^2 + \|\ex^j\|_{H^1(\Gamma_h[\bfx_*^{j}])}^2 + \|\ev^j\|_{H^1(\Gamma_h[\bfx_*^{j}])}^2 \Bigr) \notag\\
&\ +
c \sum_{j=0}^{q-1} \Bigl( \|\eu^{j}\|_{H^1(\Gamma_h[\bfx_*^{j}])}^2 + \|\ex^{j}\|_{H^1(\Gamma_h[\bfx_*^{j}])}^2 \Bigr) \notag\\
&\ + c \tau \sum_{j=1}^{n-1} \|(\ex^j -\ex^{j-1})/\tau \|_{H^1(\Gamma_h[\bfx_*^{j}])}^2 \notag\\
&\ +
c \tau \sum_{j=q}^{n}\|\du^j\|_{L^2(\Gamma_h[\bfx_*^j])}^2 
\qquad\mbox{for\, $q\le n\le m$}.
\end{align}
We note that in \cite[(10.32)]{MCF} this error bound is formulated in the $H^1(\Gamma_h[\wtx^n])$ norm. In view of \eqref{math-ind-W1inf} and the norm equivalence discussed in Section~\ref{subsec:prep}, the error bound (with different constants) is valid also for the
$H^1(\Gamma_h[\bfx_*^n])$ norm as stated here.

The stability estimates for \eqref{eq:error eq - v - full} are obtained by taking the $H^1(\Gamma_h[\bfx_*^n])$ norm of both sides of \eqref{eq:error eq - v - full}. Using \cite[Lemma~5.3]{Willmore} together with \eqref{math-ind-W1inf},  we obtain
\begin{align}\label{stability-ev}
\|\ev^n \|_{H^1(\Gamma_h[\bfx_*^n])} 
&\le c\|\eu^n\|_{H^1(\Gamma_h[\bfx_*^n])} 
+ \| \dv^n \|_{H^1(\Gamma_h[\bfx_*^n])} \quad\mbox{for\, $q\le n\le m$}.
\end{align}

\subsection{Bound for $\lambda^n$}
We recall the definition of $\delta^n$ in \eqref{delta-n}, viz.,
    $$
    \delta^n = \left| \frac1\tau  \bigl(\area_h[\bar\bfx^n] - \area_h[\bfx^{n-1}]\bigr) +
        \frac1\tau \int_{t_{n-1}}^{t_n} p^n(t)^2 \,\d t \right|,
    $$
and we define the analogous quantity along the exact solution,
\begin{equation}\label{delta-star}
    \delta^n_* = \left| \frac1\tau  \bigl(\area_h[\bar\bfx^n_*] - \area_h[\bfx^{n-1}_*]\bigr) +
        \frac1\tau \int_{t_{n-1}}^{t_n} p^n_*(t)^2 \,\d t \right|,
    \end{equation}
    where $\bar\bfx^n_*=\bfx^{n-1}_*+\tau \bar\bfv^n_*$ and $p^n_*(t)$ is the interpolation polynomial of the values 
    $\|H(.,t_{n-j})\|_{L^2(\Gamma(t_{n-j}))}$ for $j=0,\dots,q-1$. Under the smoothness assumptions of Theorem~\ref{theorem-fd-1}, we have 
    \begin{equation}
        \label{eq:delta_star bound}
        \delta^n_* \le  c \, (h^k+\tau^q),
    \end{equation}
    as follows from \eqref{area-decay} and standard interpolation error bounds.

    We prove the following lemma, which is based on Proposition~\ref{prop:lambda}.

\begin{lemma} [Bound of the Lagrange multiplier]
\label{lem:lambda-bound}
     Under condition~\eqref{math-ind} for $n\le m$, the Lagrange multiplier $\lambda^n$ exists for sufficiently small $\tau$ and $h$ and is bounded by
    \begin{align}
    \label{lambda-bound}
    &|\lambda^n| \le c \delta^n \quad\text{with}\quad 
    \\
    \nonumber
    &\delta^n \le c \sum_{j=0}^{q} \Big(  \| {\widetilde{\bf e}}_{\bf x}^{n-j} \|_{H^1(\Gamma_h[\bfx_*^{n-j}])} + \| \ev^{n-j} \|_{H^1(\Gamma_h[\bfx_*^{n-j}])} 
    + \| \eu^{n-j} \|_{L^2(\Gamma_h[\bfx_*^{n-j}])}   \Big) \notag\\
    &\hspace{21pt} + \delta^n_* + c \, (\tau^q+h^k) . 
    \end{align}
    Moreover, the simplified Newton iteration converges linearly with the convergence factor $\gamma^n=O(\tau\delta^n)$. 
\end{lemma}

\begin{proof} We split the proof into two parts, proving the bounds on $\lambda^n$ and $\delta^n$, respectively.

    (a) Since a compact closed surface cannot have zero mean curvature everywhere on the surface, we have along the exact surface that there exists a positive constant $\alpha$ such that
    $$
    -\int_{\Gamma(t)} \nabla_{\Gamma(t)}\cdot v(.,t) =  \int_{\Gamma(t)} H(.,t)^2 \ge  \alpha >0, \qquad 0\le t \le T.
    $$

Since the right-hand side of \eqref{stability-eu-2} depends only on the errors at $t_{n-j}$ with $j\ge 1$ (except the defect term), substituting \eqref{math-ind} into \eqref{stability-eu-2}--\eqref{stability-ev} yields
\begin{align}\label{eun+evn-H1}
        & \| \eu^{n} \|_{H^1(\Ga_h[\xls^n])} 
        + \| \ev^{n} \|_{H^1(\Ga_h[\xls^n])} 
        \le Ch^{\frac32} + c \, (\tau^q+h^k) 
\end{align}
for $q\leq n\leq m$. Under this condition and $\tau\leq c_0h$ and $q\ge 2$, using the inverse estimate of finite element functions, it follows that 
$$
\| \ev^{n} \|_{W^{1,\infty}(\Ga_h[\xls^n])} \le Ch^{\frac12} 
\quad\mbox{for}\,\,\, q\leq n\leq m . 
$$ 
The relation $\bar{\bf x}^n={\bf x}^{n-1}+\tau\bar {\bf v}^n$ and the result $\| \ex^{n-1} \|_{W^{1,\infty}(\Ga_h[\xls^{n-1}])} \le ch^{1/2}$ shown in \eqref{math-ind-W1inf}, imply that 
\begin{align*}
\| \bar{\bf x}^n - {\bf x}_*^n \|_{W^{1,\infty}(\Ga_h[\xls^n])} 
&\le 
\| {\bf e}_{\bf x}^{n-1} \|_{W^{1,\infty}(\Ga_h[\xls^n])} + c\tau 
+ \| {\bf x}_*^{n-1} - {\bf x}_*^n \|_{W^{1,\infty}(\Ga_h[\xls^n])} \\ 
&\le Ch^{\frac12} 
\quad\mbox{for}\,\,\, \tau\le c_0h \,\,\,\mbox{and}\,\,\, q\leq n\leq m . 
\end{align*}
For sufficiently small $h$, the above two results imply that $\int_{\Gamma_h[\bar\bfx^n]} \nabla_{\Gamma_h[\bar\bfx^n]} \cdot \bar v_h^n$ and $\int_{\Gamma(t)} \nabla_{\Gamma(t)}\cdot v(.,t) $ are sufficiently close. Therefore, condition \eqref{sigma-n} is satisfied for $n\le m$ and even with a lower bound:
$$
\sigma^n = -\int_{\Gamma_h[\bar\bfx^n]} \nabla_{\Gamma_h[\bar\bfx^n]} \cdot \bar v_h^n \ge \half \alpha >0. 
$$
Moreover, $\eta^n$ of \eqref{eta-n} is then bounded by a constant. In part (b) of the proof, we show the bound for $\delta^n$ stated in \eqref{lambda-bound}. These bounds allow us to apply Proposition~\ref{prop:lambda} with $\gamma^n=O(\tau\delta^n)$ and $\rho^n=O(\delta^n)$ to conclude that there exists a unique Lagrange multiplier $\lambda^n$ satisfying $|\lambda^n| \le \rho^n = O(\delta^n)$. 

    (b) To prove the bound for $\delta^n$ in \eqref{lambda-bound}, 
    we compare $\delta^n$ with $\delta^n_*$ and note that 
    \begin{align*}
    &\left| \frac1\tau \int_{t_{n-1}}^{t_n} p^n(t)^2 \d t - \frac1\tau \int_{t_{n-1}}^{t_n} p^n_*(t)^2 \,\d t \right| \notag\\
    &
    \le c \sum_{j=0}^q \Big(
    \| \eu^{n-j} \|_{L^2(\Gamma_h[\bfx_*^{n-j}])} 
    + \| {\widetilde{\bf e}}_{\bf x}^{n-j} \|_{L^2(\Gamma_h[\bfx_*^{n-j}])} 
    + \tau^q + h^k \Big) .
    \end{align*}
where the extra term $\| {\widetilde{\bf e}}_{\bf x}^{n-j} \|_{L^2(\Gamma_h[\bfx_*^{n-j}])}+\tau^q+h^k$ comes from the discrepancy between $\Gamma_h[\widetilde{\bf x}^{n-j}]$ and $\Gamma(t_{n-j})$ used in the definitions of $p^n(t)$ and $p_*^n(t)$, respectively.
It remains to bound the difference of the divided differences of the areas,
$$
\eps^n =\frac1\tau  \bigl(\area_h[\bar\bfx^n] - \area_h[\bfx^{n-1}]\bigr) -
\frac1\tau  \bigl(\area_h[\bar\bfx^n_*] - \area_h[\bfx^{n-1}_*]\bigr) .
$$
We introduce the surfaces, for $0\le \theta \le 1$,
\begin{align*}
    \Gamma_h^{n-1,\theta} = \Gamma_h[\bfx^{n-1}+\theta\tau \bar\bfv^{n}],
    \\
    \Gamma_{h,*}^{n-1,\theta} = \Gamma_h[\bfx_*^{n-1}+\theta\tau \bar\bfv_*^{n}],
\end{align*}
and we let $x_h^{n-1,\theta}$ and $\bar v_h^{n,\theta}$ be the finite element functions on $\Gamma_h^{n-1,\theta}$ with nodal vectors 
$\bfx^{n-1}$ and  $\bar\bfv^{n}$, respectively, and analogously $x_{h,*}^{n-1,\theta}$ and $\bar v_{h,*}^{n,\theta}$. We then have, using the Leibniz formula,
$$
\frac1\tau  \bigl(\area_h[\bar\bfx^n] - \area_h[\bfx^{n-1}]\bigr) =
\frac1\tau \int_0^1 \frac{\d}{\d \theta} \int_{\Gamma_h^{n-1,\theta}} 1 \ \d \theta =
\int_0^1  \int_{\Gamma_h^{n-1,\theta}} \nabla_{\Gamma_h^{n-1,\theta}} \cdot \bar v_h^{n,\theta} \ \d \theta.
$$
We proceed analogously for the second term in $\eps^n$ and thus obtain
$$
\eps^n = \int_0^1 \biggl( \int_{\Gamma_h^{n-1,\theta}} \nabla_{\Gamma_h^{n-1,\theta}} \cdot \bar v_h^{n,\theta}  -
\int_{\Gamma_{h,*}^{n-1,\theta}} \nabla_{\Gamma_{h,*}^{n-1,\theta}} \cdot \bar v_{h,*}^{n,\theta} \biggr) \d \theta.
$$
We introduce the further intermediate surfaces, for $0\le \theta \le 1$ and $0\le \vartheta\le 1$,
$$
    \Gamma_h^{n-1,\theta,\vartheta} = \Gamma_h[\bfx_*^{n-1}+\theta\tau \bar\bfv_*^{n}
    + \vartheta (\bfe_\bfx^{n-1} + \theta\tau \bar\bfe_\bfv^{n})],
$$
and we let $\bar v_h^{n,\theta,\vartheta}$ be the finite element function on $\Gamma_h^{n-1,\theta,\vartheta}$ with nodal vector 
$\bar\bfv_*^{n}+\vartheta \bar\bfe_\bfv^{n}$,
so that
\begin{align*}
\eps^n &= \int_0^1 \int_0^1 \frac{\d}{\d \vartheta} \int_{\Gamma_h^{n-1,\theta,\vartheta}} 
\nabla_{\Gamma_h^{n-1,\theta,\vartheta}} \cdot \bar v_{h}^{n,\theta,\vartheta} \ \d \vartheta \ \d \theta
\\
&= \int_0^1 \int_0^1 \int_{\Gamma_h^{n-1,\theta,\vartheta}} 
\Bigl( \mat_\vartheta (\nabla_{\Gamma_h^{n-1,\theta,\vartheta}} \cdot \bar v_{h}^{n,\theta,\vartheta}) \ + \Bigr.
\\
& \hspace{3.2cm}\Bigl.
(\nabla_{\Gamma_h^{n-1,\theta,\vartheta}} \cdot \bar v_{h}^{n,\theta,\vartheta})
(\nabla_{\Gamma_h^{n-1,\theta,\vartheta}} \cdot (\bar e_x^{n-1} + \theta\tau \bar e_v^{n}) \Bigr) \ \d \vartheta \ \d \theta.
\end{align*}
Using the formula for $\mat \nabla_\Gamma f$ in \cite{DziukKronerMuller} and norm equivalences as in Section~\ref{subsec:prep}, we conclude that
$$
|\eps^n| \le c \, (\| \bfe_\bfx^{n-1} \|_{H^1(\Ga_h[\bfx_*^n])} +  \| \bar \bfe_\bfv^{n} \|_{H^1(\Ga_h[\bfx_*^n])} ),
$$
which finally yields the bound of $\delta^n$ in \eqref{lambda-bound}, and completes the proof.
\qed
\end{proof}

\subsection{Stability estimate for $\ex^n$}
Inserting the bound of Lemma~\ref{lem:lambda-bound} for $\lambda^n$ into \eqref{eq:error eq - x - full} and using the equivalence of shifted norms yields
\begin{align}
    \label{stability-ex}
    \| \ex^n \|_{H^1(\Ga_h[\bfx_*^n])} &\le (1+c\tau) \| \bfe_\bfx^{n-1} \|_{H^1(\Ga_h[\bfx_*^{n-1}])} 
    \\
    \nonumber
    &\quad +
    c\tau \sum_{j=0}^q \Bigl(  \| \ev^{n-j} \|_{H^1(\Gamma_h[\bfx_*^{n-j}])} 
    +  \| \eu^{n-j} \|_{L^2(\Gamma_h[\bfx_*^{n-j}])}   \Bigr) + c\tau \delta^n_*.
\end{align}

\subsection{Combined stability estimate}
\label{section:combine}

Combining the stability estimates \eqref{stability-eu-2}, \eqref{stability-ev}, \eqref{lambda-bound} and \eqref{stability-ex}, and using a discrete Gronwall inequality, norm equivalences and the induction argument in the same way as in \cite[Section~10, Part (C)]{MCF}, we arrive at the following stability result, which bounds the errors in terms of defects and initial errors. Up to different constants and the extra defects $\delta^n_*$ of \eqref{delta-star}, this is the same as in \cite[Proposition 10.1]{MCF}.

\begin{proposition}[Stability]
\label{proposition:stability - coupled problem - full} 
 Consider the area-decreasing full discretization \eqref{eq:multistep area-decreasing}.
 Assume that there exists $\kappa$ with $1 < \kappa \leq k$ such that the defects are bounded by 
        \begin{align}
        \nonumber
        &\|\dx^n\|_{H^1(\Ga_h[\bfx_*^n])} \leq c h^\kappa , \quad
		\|\dv^n\|_{H^1(\Ga_h[\bfx_*^n])} \leq c h^\kappa , \quad 
		\|\du^n\|_{L^2(\Ga_h[\bfx_*^n])} \leq c h^\kappa ,
        \\
        \label{eq:assume small defects}
        &\delta^n_* \le c h^\kappa \quad\text{with $\delta^n_*$ of \eqref{delta-star}}, 
        \end{align}
	 for $q\tau \leq n\tau \leq T$, and that also the errors of the starting values are bounded by
	\begin{equation}
	\label{eq:assume small initial values}
		\|\ex^i\|_{H^1(\Ga_h[\bfx_*^i])} \leq c h^\kappa , \quad
		\|\ev^i\|_{H^1(\Ga_h[\bfx_*^i])} \leq c h^\kappa , \quad 
		\|\eu^i\|_{H^1(\Ga_h[\bfx_*^i])} \leq c h^\kappa , \quad
	\end{equation} 
	for $i=0,\dots,q-1$,  and, with the notation $\be \ex^i = (\ex^i-\ex^{i-1})/\tau$, for $i = 1,\dotsc,q-1$, 
	\begin{equation}
	\label{eq:assume small initial values 2}
		\tau^{1/2} \| \be \ex^i \|_{H^1(\Ga_h[\bfx_*^i])}  \leq c h^\kappa .
	\end{equation} 
Then, there exists $h_0>0$ such that the following stability estimate holds for all $h\leq h_0$ and $\tau \leq c_0 h$, and all $n$ with $n\tau \le T$, 
	\begin{equation}
	\label{eq:stability bound - full}
		\begin{aligned}
			& \| \ex^n \|_{H^1(\Ga_h[\bfx_*^n])}^2 +  \| \ev^n \|_{H^1(\Ga_h[\bfx_*^n])}^2 +  \| \eu^n \|_{H^1(\Ga_h[\bfx_*^n])}^2 
            + |\lambda^n|^2 \\
			& \leq \ C  \sum_{i=0}^{q-1}\Bigl( \| \ex^i \|_{H^1(\Ga_h[\bfx_*^i])}^2 + \| \ev^i \|_{H^1(\Ga_h[\bfx_*^i])}^2 + \| \eu^i \|_{H^1(\Ga_h[\bfx_*^i])}^2\Bigr)  + C \tau \sum_{i=1}^{q-1} \| \be \ex^i \|_{H^1(\Ga_h[\bfx_*^i])}^2  \\
			& \quad +C \max_{0\le j \le n}\|\dv^j\|_{H^1(\Ga_h[\bfx_*^j])}^2  
			+ C \tau\sum_{j=q}^{n} \|\du^j\|_{L^2(\Ga_h[\bfx_*^j])}^2 + C  \tau\sum_{j=q}^{n} \|\dx^j\|_{H^1(\Ga_h[\bfx_*^j])}^2
            \\
            & \quad +C\tau \sum_{j=q}^n |\delta_*^j|^2,
		\end{aligned}
	\end{equation}
	where $C$ is independent of $h$, $\tau$ and $n$ with $n\tau\le T$, but depends on the final time $T$.
\end{proposition}

\subsection{Error bound of Theorem~\ref{theorem-fd-1}}
Together with the known $O(h^k+\tau^q)$ bounds for the defects \cite{MCF}, the bound \eqref{eq:delta_star bound}, and the interpolation and Ritz projection errors in the case of smooth solutions, Proposition~\ref{proposition:stability - coupled problem - full} yields the error bounds of Theorem~\ref{theorem-fd-1} for $q\leq n\leq m$ in the same way as is shown in \cite[Section~12]{MCF}. For sufficiently small $h$ and $\tau\leq c_0h$, applying the inverse estimates of finite element functions recovers \eqref{math-ind} for $n=m+1$. This proves the mathematical induction on \eqref{math-ind} and therefore completes the proof of Theorem~\ref{theorem-fd-1}.\qed

\section{Proof of Theorem \ref{theorem-fd-2}}

The proof of Theorem~\ref{theorem-fd-2} follows the same line as that of Theorem~\ref{theorem-fd-1}, except for the stability estimate for the velocity. We therefore highlight this difference and omit the remaining, essentially identical arguments. 

\subsection{Reformulation of the MDR equations}
We define a discrete Laplace--Beltrami operator $\Delta_{h,\Gamma_h[\bfx]}: S_h[ {\bf x}]\to S_h[ {\bf x}]$ via duality by   
\begin{equation*}
\int_{\Gamma_h[\bfx]} \Delta_{h,\Gamma_h[ {\bf x}]} v_h \, w_h  = -\int_{\Gamma_h[\bfx]} \nabla_{ \Gamma_h[ {\bf x}]} v_h \cdot \nabla_{ \Gamma_h[ {\bf x}]} w_h  \quad \forall w_h\in S_h[\bfx] . 
\end{equation*}
Let $P_{\Gamma_h[\bfx]} : L^2(\Gamma_h[\bfx])\rightarrow S_h[\bfx]$ be the $L^2$-orthogonal projection onto the finite element space.
Then, replacing the test function $\chi_h $ in \eqref{eq:MCF-MDR-eq3} by $-\Delta_{h,\Gamma_h[\bfx]}\chi_h$, we obtain 
\begin{align}\label{grad-v.n.eq}
&\hspace{-10pt} \int_{\Gamma_h[\bfx]} \nabla_{\Gamma_h[\bfx]}  (v_h \cdot n_h)
\cdot  \nabla_{\Gamma_h[\bfx]} \chi_h \notag\\ 
=& - \int_{\Gamma_h[\bfx]} \nabla_{\Gamma_h[\bfx]} H_h \cdot \nabla_{\Gamma_h[\bfx]} \chi_h \notag \\
& + \int_{\Gamma_h[\bfx]} \nabla_{\Gamma_h[\bfx]} ( v_h \cdot n_h- P_{\Gamma_h[\bfx]} (v_h \cdot n_h))
\cdot  \nabla_{\Gamma_h[\bfx]} \chi_h . 
\end{align}
This weak formulation can be written as the following matrix-vector form:
\begin{align}
\label{Mform-Dv}
{\bf D}(\bfx,\bfn)\bfv &= -{\bf A}(\bfx){\bf H} +{\boldsymbol{\rho}}(\bfx,\bfn,\bfv) . 
\end{align}
where ${\boldsymbol{\rho}}(\bfx,\bfn,\bfv)\in \R^\dof$ is the nodal vector satisfying the following relation:  
\begin{equation}\label{IntroL2}
{\boldsymbol{\rho}}(\bfx,\bfn,\bfv)\cdot {\boldsymbol{\chi}}
= \int_{\Gamma_h[\bfx]} \nabla_{\Gamma_h[\bfx]} [ v_h \cdot n_h- P_{\Gamma_h[\bfx]} (v_h \cdot n_h)]
\cdot \nabla_{\Gamma_h[\bfx]} \chi_h 
\end{equation} 
for all $ \chi_h\in S_h[\bfx]$ with nodal vector $\boldsymbol{\chi}\in\R^\dof$, and 
$$
\boldsymbol{\chi}^T D(\mathbf{x},\mathbf{n})\boldsymbol{\psi}
=
\int_{\Gamma_h[\mathbf{x}]}
\nabla_{\Gamma_h[\mathbf{x}]}(\psi_h\cdot n_h)
\cdot
\nabla_{\Gamma_h[\mathbf{x}]}\chi_h
\qquad
\forall \boldsymbol{\chi}\in\mathbb{R}^N,\ 
\boldsymbol{\psi}\in\mathbb{R}^{3N}  
$$
for all $ \chi_h\in S_h[\bfx]$ and $\psi_h\in S_h[\bfx]^3$ with nodal vectors $\boldsymbol{\chi}\in\R^\dof$ and $\boldsymbol{\psi}\in\R^{3\dof}$, respectively. 
The system in \eqref{eq:MCF-matrix-vector-MDR} is used in the implementation, whereas \eqref{Mform-Dv} is a direct consequence of \eqref{eq:MCF-MDR-eq3} and will be employed in the error analysis. 

\subsection{Defects and error equations}
Let $u_h^*(t) \in S_h[\xs(t)]^4$ be the Ritz projection of $u$,  and let $v_h^*(t) \in S_h[\xs(t)]^3$ be the Ritz projection of $v$. Their nodal vectors are denoted by $\bfu^*(t)$ and $\bfv^*(t)$, respectively. For the fully discrete FEM with tangential motion defined in \eqref{eq:multistep-MDR}, the interpolated values $\xls^n={\bfx}^*(t_n)$ and $\boldsymbol{\mu}_*^n=\boldsymbol{\mu}^*(t_n)$, and the Ritz projections $\vls^n={\bfv}^*(t_n)$ and $\uls^n={\bfu}^*(t_n)$, satisfy the following matrix-vector equations: 
\begin{subequations}\label{eq:MDR-defects}
\begin{align} 
    \bfM(\wtx_*^n) {\bf\dot\bfu}_*^n + \bfA(\wtx_*^n) \bfu_*^n = &\  \bff(\wtx_*^n,\wtu_*^n,\wtv_*^n) + \bfM(\wtx_*^n) \du^n , \\[3pt]
    {\bf A}(\wtx_*^n)\bfv_*^n + {\bf B}(\wtx_*^n,{\bfn}_*^n)^{\rm T}\boldsymbol{\mu}_*^n = &\ \bfM(\wtx_*^n){\bf d}_{\boldsymbol{\mu}}^n , \\ 
        {\bf B}(\wtx_*^n,{\bfn}_*^n) \bfv_*^n = &\ - {\bf M}(\wtx_*^n)\bfH_*^n + \bfM(\wtx_*^n){\bf d}_{\bf v}^n , \\
    \frac{1}{\tau} \big( \bfx_*^n - \bfx_*^{n-1} \big) = &\ \bar\bfv_{*}^n + {\bfd}_{\bfx}^n , 
\end{align}
\end{subequations}
where defect terms ${\bfd}_{\bfx}^n$ and $\du^n$ have the same estimates as that in the method without tangential motion (though it is caused by the Ritz projection now):
\begin{align}
\| \du^n \|_{L^2(\Gamma_h[\bfx_*^n])} 
+ \| {\bfd}_{\bfx}^n \|_{H^1(\Gamma_h[\bfx_*^n])} \le c \, (\tau^q+h^k) , 
\end{align} 
while the defect terms  $\dv^n$ and ${\bf d}_{\boldsymbol{\mu}}^n$ are defined in \cite[(3.9b) and (3.9c)]{Hu-Li-2022}, satisfying the following estimates: 
\begin{align}
\| \dv^n \|_{L^2(\Gamma_h[\bfx_*^n])} + \| {\bf d}_{\boldsymbol{\mu}}^n\|_{L^2(\Gamma_h[\bfx_*^n])} \le c \, (\tau^q+h^k) . 
\end{align} 
This estimate was established in \cite[Section~3.6]{Hu-Li-2022} for the semi-discrete (space-discretized) scheme. In our fully discrete setting, we include an additional factor $\tau^q$ to account for the truncation error introduced by the time discretization.

Similarly, $(\bfx_*^n,\bfu_*^n,\bfv_*^n)$ satisfies the following relation corresponding to \eqref{Mform-Dv}:
\begin{align}
\label{Mform-dv}
{\bf D}(\widetilde{\bfx}_*^n,\bfn_*^n)\bfv_*^n &= -{\bf A}(\widetilde{\bfx}_*^n){\bf H}_*^n +{\boldsymbol{\rho}}(\widetilde{\bfx}_*^n,\bfn_*^n,\bfv_*^n) + \bfM(\widetilde{\bfx}_*^n) \widetilde \bfd_{\bfv}^n. 
\end{align} 
where the defect term, $\widetilde \bfd_\bfv^n$, is a nodal vector satisfying the following relation, obtained by substituting the nodal vectors ${\bf v}_*^n$, ${\bf n}_*^n$, ${\bf H}_*^n$ and $\widetilde{\bfx}_*^n$ (viewed as finite element functions on $\Gamma_h[\widetilde{\bfx}_*^n]$) into \eqref{grad-v.n.eq}: 
\begin{align}\label{tdv-weak} 
(\bfM(\widetilde{\bfx}_*^n) \widetilde \bfd_\bfv^n)^{\rm T} \boldsymbol{\chi} 
:= &\ \int_{\Gamma_h[\widetilde{\bfx}_*^n]} \nabla_{\Gamma_h[\widetilde{\bfx}_*^n]} P_{\Gamma_h[\widetilde{\bfx}_*^n]} ({\bf v}_*^n \cdot {\bfn}_*^n)
\cdot  \nabla_{\Gamma_h[\widetilde{\bfx}_*^n]} \boldsymbol{\chi}  \notag \\ 
& + \int_{\Gamma_h[\widetilde{\bfx}_*^n]} \nabla_{\Gamma_h[\widetilde{\bfx}_*^n]} {\bf H}_*^n \cdot \nabla_{\Gamma_h[\widetilde{\bfx}_*^n]} \boldsymbol{\chi} 
\quad\forall\, \boldsymbol{\chi}\in\R^\dof. 
\end{align}
The following estimate holds: 
\begin{align}
\| \widetilde{\bfd}_{\bfv}^n \|_{H^{-1}(\Gamma_h[\widetilde{\bfx}_*^n])} \le c \, (\tau^q+h^k) . 
\end{align} 
Again, this result was established in \cite[Section~3.6]{Hu-Li-2022} for the semi-discrete (space-discretized) scheme. In our fully discrete setting, we include an additional factor $\tau^q$ to account for the truncation error introduced by the time discretization.

By considering the difference between \eqref{eq:MDR-defects} and \eqref{eq:multistep-MDR}, we obtain the following equations for the error functions $\ex^n =  \bfx^n - \xls^n$, $\ev^n = \bfv^n - \vls^n$, ${\bf e}_{\boldsymbol{\mu}}^n = \boldsymbol{\mu}^n - \boldsymbol{\mu}_*^n $ and $\eu^n = \bfu^n - \uls^n$: 
\begin{subequations}\label{stability-form2}
\begin{align}
        \label{MDR:error-u-full}
        \bfM(\widetilde\bfx^n){\bf\dot\bfe}_{\bfu}^n+ \bfA(\widetilde \bfx^n)\eu^n 
        = &\ -  \big( \bfM(\widetilde \bfx^n)-\bfM(\widetilde \bfx_*^n) \big) {\bf\dot\bfu}_*^n \notag \\[3pt]
        &\ - \big( \bfA(\widetilde \bfx^n)-\bfA(\widetilde \bfx_*^n) \big) \bfu_*^n \notag \\[3pt]
        &\ + \big(\bff(\widetilde \bfx^n,\widetilde \bfu^n,\widetilde \bfv^n) - \bff(\widetilde \bfx_*^n,\widetilde \bfu_*^n,\widetilde \bfv_*^n)\big) \notag\\
        &\ -  \bfM(\wtxls^n)  \du^n \\[2pt]
        \label{MDR:error-v-full}
        {\bf B}(\widetilde{\bfx}_*^n,\bfn_*^n)^{\rm T} \bfe_{\boldsymbol{\mu}}^n + {\bf A}(\widetilde{\bfx}_*^n)\bfe_\bfv^n
        =& \ ({\bf B}(\widetilde{\bfx}_*^n,\bfn_*^n)^{\rm T}-{\bf B}(\widetilde{\bfx}^n,\bfn^n)^{\rm T}){\boldsymbol{\mu}}^n \notag\\
        & + ({\bf A}(\widetilde{\bfx}_*^n)-{\bf A}(\widetilde{\bfx}^n))\bfv^n - {\bf M}(\widetilde{\bfx}_*^n) \bfd_{\boldsymbol{\mu}}^n ,\\[5pt] 
        \label{MDR:error-k-full}
        {\bf B}(\widetilde{\bfx}_*^n,\bfn_*^n)\bfe_\bfv^n 
        =&\ ({\bf B}(\widetilde{\bfx}_*^n,\bfn_*^n)-{\bf B}(\widetilde{\bfx}^n,\bfn^n))\bfv^n \notag \\
         &\ - {\bf M}(\widetilde{\bfx}_*^n)\bfe_\bfH^n  + ({\bf M}(\widetilde{\bfx}_*^n)-{\bf M}(\widetilde{\bfx}^n))\bfH^n \notag \\
         &\ - {\bf M}(\widetilde{\bfx}_*^n) \bfd_{\bfv}^n, \\[3pt] 
        \label{MDR:error-x-full}
        \frac{1}{\tau} \big( \ex^n - \ex^{n-1} \big) = &\  \evb^n + \lambda^n \bar\bfv^n  - {\bfd}_{\bfx}^n . 
\end{align}
\end{subequations}
Similarly, the error equation corresponding to \eqref{Mform-dv} is given by 
\begin{align}
\label{error-ev-grad}
{\bf D}(\widetilde{\bfx}_*^n,\bfn_*^n)\bfe_\bfv^n =& \ ({\bf D}(\widetilde{\bfx}_*^n,\bfn_*^n)-{\bf D}(\widetilde{\bfx}^n,\bfn^n))\bfv^n \notag\\
&\ 
- {\bf A}(\widetilde{\bfx}_*^n)\bfe_\bfH^n + ({\bf A}(\widetilde{\bfx}_*^n)-{\bf A}(\widetilde{\bfx}^n))\bfH^n \notag \\
&\ - \big(\boldsymbol{\rho}(\widetilde{\bfx}_*^n,\bfn_*^n,\bfv_*^n) - \boldsymbol{\rho}(\widetilde{\bfx}^n,\bfn^n,\bfv^n)\big) - {\bf M}(\widetilde{\bfx}_*^n) \widetilde \bfd_{\bfv}^n .
\end{align}

\subsection{Stability estimates}
The stability estimate for \eqref{MDR:error-u-full} is the same as \eqref{eq:error eq - u - full} as they have the same structure. Therefore, \eqref{stability-eu-2} still holds, i.e., 
\begin{align}\label{stability-MDR-eu-2}
\|\eu^n\|_{H^1(\Gamma_h[\bfx_*^n])}^2 
    \le &\ c\sum_{j=0}^{q-1} \| \ex^{n-1-j} \|_{H^1(\Gamma_h[\bfx_*^{n-1-j}])}^2 \notag\\
  &\ +  c\tau\sum_{j=q}^{n-1} \Bigl( \|\eu^j\|_{H^1(\Gamma_h[\bfx_*^{j}])}^2 + \|\ex^j\|_{H^1(\Gamma_h[\bfx_*^{j}])}^2 + \|\ev^j\|_{H^1(\Gamma_h[\bfx_*^{j}])}^2 \Bigr) \notag\\
&\ +
c \sum_{j=0}^{q-1} \Bigl( \|\eu^{j}\|_{H^1(\Gamma_h[\bfx_*^{j}])}^2 + \|\ex^{j}\|_{H^1(\Gamma_h[\bfx_*^{j}])}^2 \Bigr) \notag\\
&\ + c \tau \sum_{j=1}^{n-1} \|(\ex^j -\ex^{j-1})/\tau \|_{H^1(\Gamma_h[\bfx_*^{j}])}^2 \notag\\
&\ + 
c \tau \sum_{j=q}^{n} \|\du^j\|_{L^2(\Gamma_h[\bfx_*^j])}^2 
\qquad\mbox{for\, $q\le n\le m$}.
\end{align}

The following stability estimate for \eqref{MDR:error-v-full}-\eqref{MDR:error-k-full} has been established in \cite[inequality (3.78)]{Hu-Li-2022}:
\begin{align*}
\|\bfe_\bfv^n\|_{H^1(\Gamma_h[\widetilde{\bfx}_*^n])} 
& \le c \, (\|\bfe_\bfu^n\|_{H^1(\Gamma_h[\widetilde{\bfx}_*^n])} + \|\widetilde{\bfe}_{\bfx}^n\|_{H^1(\Gamma_h[\widetilde{\bfx}_*^n])} + \| \bfd_\bfv^n\|_{L^2(\Gamma_h[\widetilde{\bfx}_*^n])}) \notag\\
&\quad\
+ c \, (\| {\widetilde\bfd}_\bfv^n\|_{H^{-1}(\Gamma_h[\widetilde{\bfx}_*^n])} 
+ \| \bfd_{\boldsymbol{\mu}}^n\|_{L^2(\Gamma_h[\widetilde{\bfx}_*^n])} ) .
\end{align*}
In view of \eqref{math-ind-W1inf} and the norm equivalence discussed in Section~\ref{subsec:prep}, this error bound (with different constants) is valid also on the surface $\Gamma_h[{\bfx}_*^n]$, i.e.,
\begin{align}\label{stability-MDR-ev}
\|\bfe_\bfv^n\|_{H^1(\Gamma_h[{\bfx}_*^n])} 
& \le c \, (\|\bfe_\bfu^n\|_{H^1(\Gamma_h[{\bfx}_*^n])} + \sum_{j=1}^q\|\bfe_{\bfx}^{n-j}\|_{H^1(\Gamma_h[{\bfx}_*^n])} + \| \bfd_\bfv^n\|_{L^2(\Gamma_h[{\bfx}_*^n])}) \notag\\
&\quad\
+ c \, (\| {\widetilde\bfd}_\bfv^n\|_{H^{-1}(\Gamma_h[\widetilde{\bfx}_*^n])} 
+ \| \bfd_{\boldsymbol{\mu}}^n\|_{L^2(\Gamma_h[{\bfx}_*^n])} ) .
\end{align} 

The stability estimate for~\eqref{MDR:error-x-full} is identical to that for \eqref{eq:error eq - x - full}, since the two relations have the same structure. Moreover,
the stability bound for \(e_\lambda^n\) is unchanged as in Lemma \ref{lem:lambda-bound}. Consequently, the estimate 
\eqref{stability-ex} remains valid for the scheme with
tangential motion, i.e.,
\begin{align}
    \label{stability-ex-tan}
    \| \ex^n \|_{H^1(\Ga_h[\bfx_*^n])} &\le (1+c\tau) \| \bfe_\bfx^{n-1} \|_{H^1(\Ga_h[\bfx_*^{n-1}])} 
    \nonumber\\
    &\quad +
    c\tau \sum_{j=0}^q \Bigl(  \| \ev^{n-j} \|_{H^1(\Gamma_h[\bfx_*^{n-j}])} 
    +  \| \eu^{n-j} \|_{L^2(\Gamma_h[\bfx_*^{n-j}])}   \Bigr) \nonumber\\
    &\quad + c\tau (h^k+\tau^q).
\end{align}
The remainder of the argument proceeds exactly as in Section \ref{section:combine}: combining 
\eqref{stability-MDR-eu-2}, \eqref{stability-MDR-ev} and \eqref{stability-ex-tan} yields the desired error bound. This
completes the proof of Theorem~\ref{theorem-fd-2}.
\hfill\qed






\section{Numerical experiments}
\label{sec:numerics}

We performed numerical simulations and experiments for  mean curvature flow using various algorithms of this paper. We report on:
\begin{itemize}
    \item[-] some important details on implementation;
	\item[-] convergence experiments with and without a Lagrange-multiplier to illustrate the theoretical results of Theorem~\ref{theorem-fd-1} and \ref{theorem-fd-2};
	\item[-] efficiency plots comparing the computational overhead for the Lagrange-multiplier algorithm.
\end{itemize}

The numerical experiments use quadratic evolving surface finite elements, implemented in the Matlab package $\ell$FEM \cite{ellFEM}. For the computation of all finite element matrices and vectors it uses high-order quadratures so that the resulting quadrature error does not feature in the discussion of the accuracies of the schemes. 
The temporal discretization uses linearly-implicit BDF--Adams methods of order $1,\dotsc,5$. 
The initial meshes were all generated using DistMesh \cite{distmesh}, without taking advantage of any symmetry of the surfaces.

\subsection{Implementation}
\label{section:implementation}

Let us recall the fully-discrete algorithm \eqref{eq:multistep area-decreasing} determining $\bfx^n$, $\bfv^n$, and $\bfu^n = (\bfn^n, \bfH^n)$.

(1) We determine the (normalised) extrapolations:
\begin{equation*}
    \widehat{\bfn}^n =  \bfn^n / |\bfn^n| 
    \andquad 
    \wtu^n = \big( \widehat{ {\widetilde \bfn}}^n  , \widetilde \bfH^n \big) .
\end{equation*}

(2) We solve the linear system for the geometric variable $\bfu^n$:
\begin{equation*}
    \Big( \delta_0 \bfM(\wtx^n) + \tau \bfA(\wtx^n) \Big) \bfu^n =  \tau \bff(\wtx^n,\wtu^n) - \bfM(\wtx^n) \sum_{j=1}^q \delta_j \bfu^{n-j}.
\end{equation*}

(3) We compute the velocity $\bfv^n$ (using a normalisation once again):
\begin{equation*}
    \bfv^n = - \bfH^n \bullet \widehat{\bfn}^n .
\end{equation*}

(4) We determine $\lambda^n \in \R$ via a (simplified) Newton iteration, which in turn gives $\bfx^n$: 
To this end, we start by rewriting \eqref{eq:multistep area-decreasing - x}, separating Lagrange-parameter-dependent parts, for arbitrary $\lambda \in \R$:
\begin{align*} 
    \bfx^n(\lambda) := &\ \bfx^n(0) + \lambda \tau {\bar\bfv}^n , \\
    \text{with} \qquad &\ \bfx^n(0) := \bfx^{n-1} + \tau {\bar\bfv}^n \andquad {\bar\bfv}^n := \sum_{j=0}^{q} \beta_j \bfv^{n-j} ,
\end{align*}
in particular
\begin{align*}
    &\ \bfx^n := \bfx^n(\lambda^n), \quad \text{for a $\lambda^n \in \R $ satisfying} ~ \area^n(\lambda^n) = \area^{n-1} - b^n , \\
    &\ \text{with} ~ \area^n(\lambda) := \int_{\Ga_h[\bfx^n(\lambda)]} 1 \andquad b^n :=  \int_{t_{n-1}}^{t_n} p^n(t)^2 \d t \, .
\end{align*}
The integral of the Lagrange interpolation polynomial $p^n$ is computed exactly, via a high-order Gauss quadrature.

The simplified Newton method for $\lambda^n$ reads:
\begin{align*}
    \lambda^{(k+1)} = &\ \lambda^{(k)} + \Delta \lambda^{(k)} ,\\
    \pa_\lambda \area^n( 0 ) \, \Delta \lambda^{(k)} = &\ - \big( \area^n( \lambda^{(k)} ) - ( \area^{n-1} - b^n ) \big) ,
\end{align*}
where the $\lambda$-derivative of $\area^n\colon \R \to \R$ is given, using the Leibniz formula, by
\begin{align*}
    \pa_\lambda \area^n( \lambda ) = &\ \pa_\lambda \int_{\Ga_h[\bfx^n(\lambda)]} 1 = \int_{\Ga_h[\bfx^n(\lambda)]} \nb_{\Ga_h[\bfx^n(\lambda)]} \cdot \big( \tau v_{\beta,h}^n \big) , \\
    \pa_\lambda \area^n( 0 )  = &\ \pa_\lambda \area^n( \lambda )|_{\lambda = 0} = \tau \int_{\Ga_h[\bfx^n(0)]} \nb_{\Ga_h[\bfx^n(0)]} \cdot v_{\beta,h}^n ,
\end{align*}
where $v_{\beta,h}^n$ denotes the appropriate finite element function with nodal values ${\bar\bfv}^n$.

We use the stopping criteria $\big| \area^n( \lambda^{(k)} ) - ( \area^{n-1} - b^n ) \big|\leq \delta b^n$ with a small $\delta \in (0,1)$.

\subsection{Comparing the algorithms without tangential motion with and without Lagrange-multiplier}

We perform numerical experiments for mean curvature flow of a sphere, following \cite[Section~13.1]{MCF}.

We report on convergence experiments for the unit sphere as initial surface $\Ga^0$. We start the algorithms from the nodal interpolations of the exact initial values $\Ga^0$, $\nu^0$, and $H^0 = 2$, subsequent initial values $i=1,\dotsc,q-1$ were chosen to be the exact values. 
In order to illustrate the convergence results of Theorem~\ref{theorem-fd-1}, we have computed the errors between the numerical solution \eqref{eq:multistep area-decreasing} and the exact solution of \eqref{KLL-form}.

\begin{figure}[htbp]
	\centering
	\includegraphics[width=\textwidth, trim={0 0 0 0}, clip]{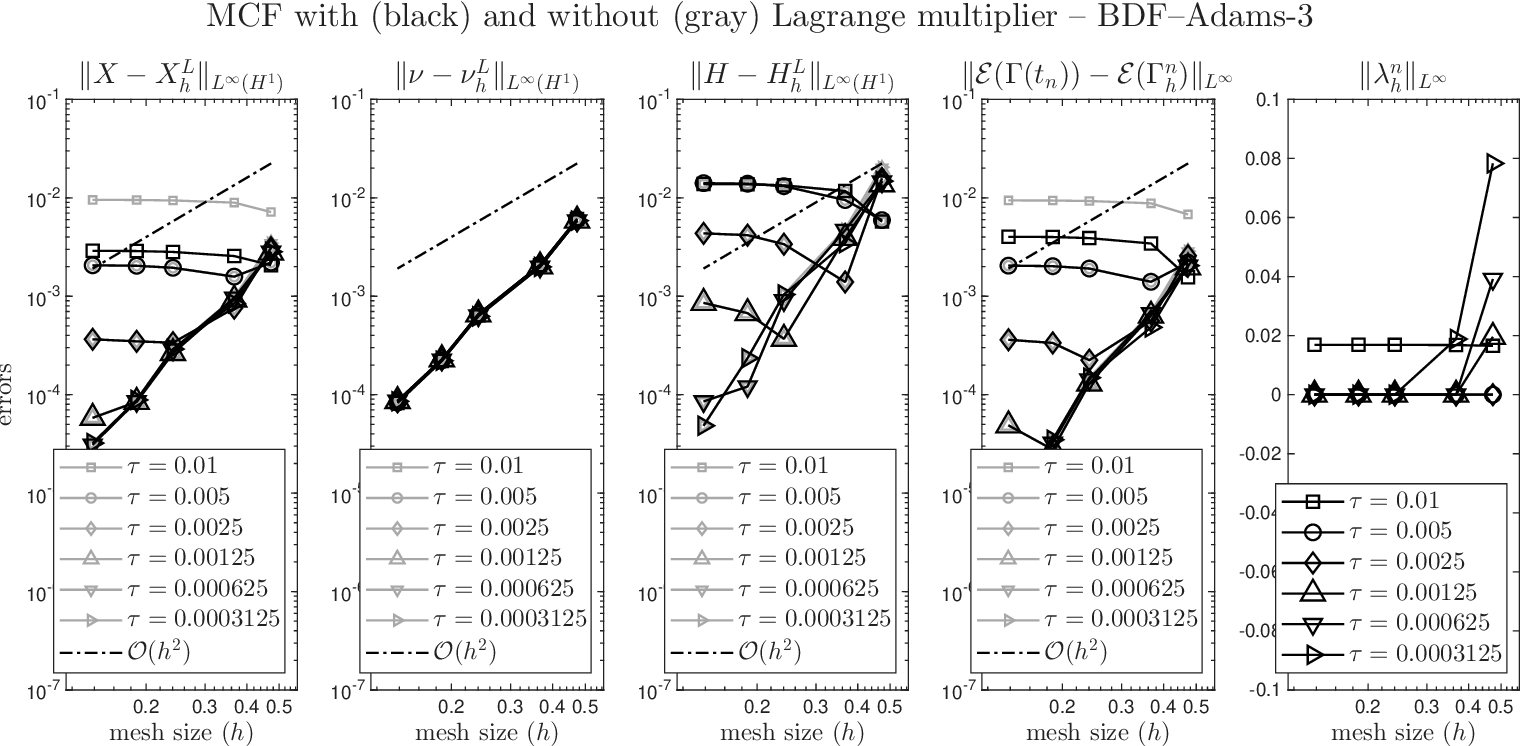}
    \caption{Spatial convergence of the BDF--Adams-3 / quadratic ESFEM discretization of mean curvature flow of a sphere with $T = 0.24$.}
	\label{fig:MCF - comp convplot space}
\end{figure}

\begin{figure}[htbp]
	\centering
	\includegraphics[width=\textwidth, trim={0 0 0 0}, clip]{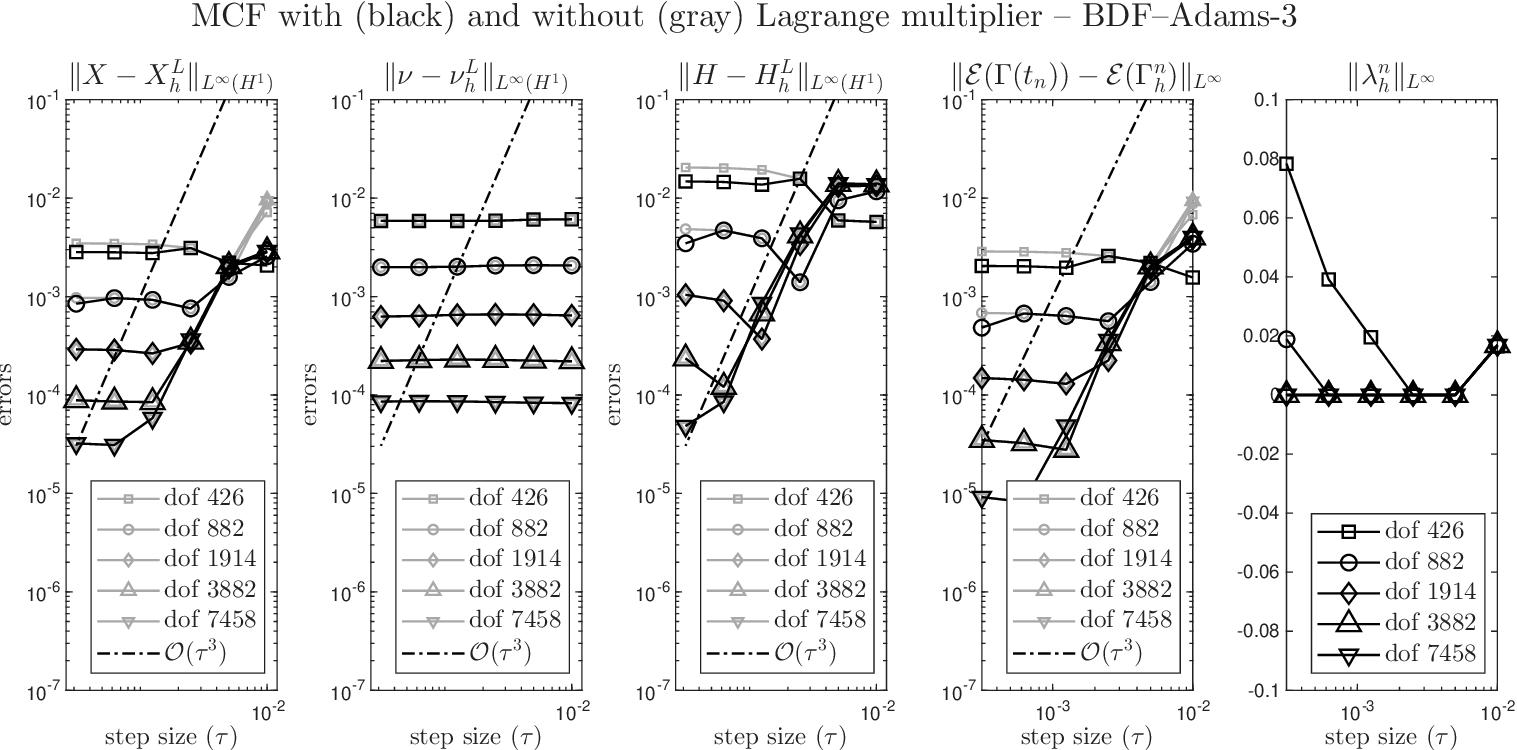}
    \caption{Temporal convergence of the BDF--Adams-3 / quadratic ESFEM discretization of mean curvature flow of a sphere with $T = 0.24$.}
	\label{fig:MCF - comp convplot time}
\end{figure}

In Figure~\ref{fig:MCF - comp convplot space} and \ref{fig:MCF - comp convplot time} we report the errors between the numerical solution and the interpolation of the exact solution until the final time $T=0.24$, for a sequence of meshes (see plots) and for a sequence of time steps $\tau_{k+1} = \tau_k / 2$. 
The gray lines correspond to the algorithm \emph{without} Lagrange multiplier \eqref{eq:MCF matrix-vector}, the black lines correspond to the algorithm \emph{with} Lagrange multiplier \eqref{eq:multistep area-decreasing}. 
The lines marked with different symbols correspond to different time step sizes and to different mesh refinements in Figure~\ref{fig:MCF - comp convplot space} and \ref{fig:MCF - comp convplot time}, respectively.
The double-logarithmic plots in the first four panels report on the $L^\infty(H^1)$ norm of the errors against the mesh width $h$ in Figure~\ref{fig:MCF - comp convplot space}, and against the time step size $\tau$ in Figure~\ref{fig:MCF - comp convplot time}.
The semi-logarithmic plot in the rightmost panel reports on the $L^\infty$-norm of the Lagrange-multiplier.

In Figure~\ref{fig:MCF - comp convplot space} we can observe two regions: a region where the spatial discretization error dominates, matching the $\calO(h^2)$ order of convergence of Theorem~\ref{theorem-fd-1} and \cite[Theorem~6.1]{MCF} (see the reference lines), and a region, with small mesh size, where the temporal discretization error dominates (the error curves flatten out). For Figure~\ref{fig:MCF - comp convplot time}, the same description applies, but with reversed roles, the errors matching the $\calO(\tau^3)$ order of convergence of Theorem~\ref{theorem-fd-1} and \cite[Theorem~6.1]{MCF} (see the reference lines).

\begin{figure}[htbp]
	\centering
	\includegraphics[width=\textwidth, trim={0 0 0 20}, clip]{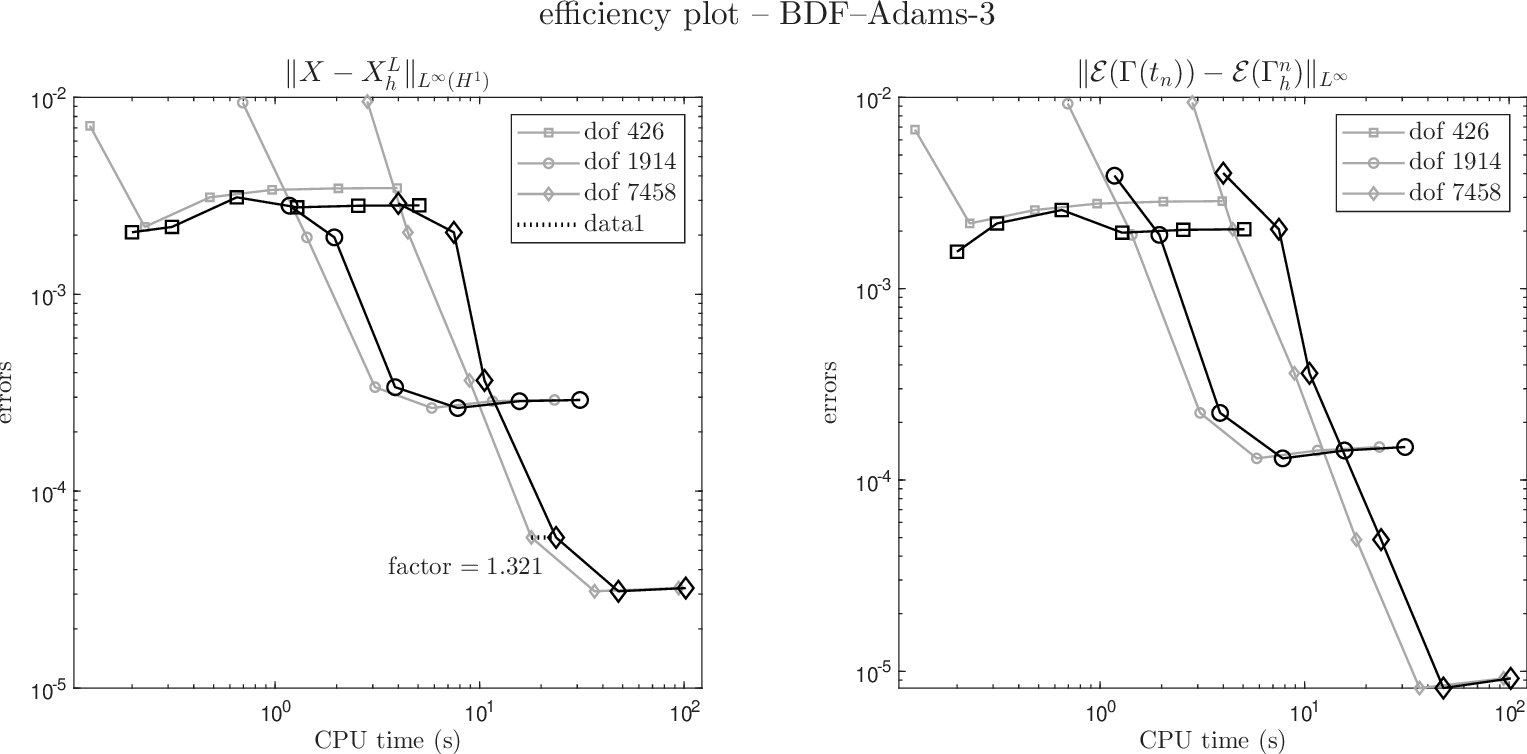}
	\caption{Efficiency plots for the BDF--Adams-3 / quadratic ESFEM discretization of mean curvature flow of a sphere with $T = 0.24$.}
	\label{fig:MCF - comp efficiency}
\end{figure}

In Figure~\ref{fig:MCF - comp efficiency} we report on computational efficiency, that is, we plot the errors in the surface and area against the CPU times required by both algorithms. As before, the gray and black lines respectively correspond to the algorithm \emph{without} and \emph{with} Lagrange multiplier. The lines correspond to various spatial refinements (see plots) while the nodes on each line correspond to time step sizes $\tau_i$ (from the convergence test). 
In Figure~\ref{fig:MCF - comp efficiency} it can be observed that the new structure-preserving algorithm \eqref{eq:multistep area-decreasing} only requires a slight overhead when compared to algorithm \eqref{eq:MCF matrix-vector} without Lagrange-multiplier. Furthermore, it is seen that this computational overhead is independent of mesh size and time step size.

%

\section*{Acknowledgements}
The foundations of this work have been laid when the three authors were participating in the Oberwolfach Research Fellows Programme at the Mathematisches Forschungsinstitut Oberwolfach (MFO) in March 2025.

The exchange between Germany and Hong Kong was funded in part by the Germany/Hong Kong Joint Research Scheme sponsored by the Research Grants Council of Hong Kong and the German Academic Exchange Service of Germany (Ref. No. G-PolyU501/24 and DAAD PPP Projekt-ID: 57750652). 

The work of B.~Kov\'{a}cs was funded by the DFG Heisenberg Programme -- Project-ID 446431602, and by the DFG Research Unit FOR 3013 \textit{Vector- and tensor-valued surface PDEs} (BA2268/6–1). 

The work of B.~Li was funded in part by the National Natural Science Foundation of China (NSFC Project No. 12525111) and the Research Grants Council of Hong Kong (PolyU/RFS2324-5S03). 

\bibliographystyle{abbrv}
\bibliography{MCF_lagrange_literature}

\end{document}